\documentclass{amsart}          

\usepackage{amsmath,amsthm}     
\usepackage{amssymb,bbm}            
\usepackage{euscript}           
\usepackage{array,enumerate,calc,graphicx}
\usepackage[matrix,arrow,curve,frame]{xy}    
\usepackage{tikz-cd}
\usepackage{bbm}

\xymatrixcolsep{1.9pc}                          
\xymatrixrowsep{1.9pc}
\newdir{ >}{{}*!/-5pt/\dir{>}}                  

\newtheorem{thm}[subsection]{Theorem}
\newtheorem{defn}[subsection]{Definition}
\newtheorem{prop}[subsection]{Proposition}

\newtheorem{cor}[subsection]{Corollary}
\newtheorem{lemma}[subsection]{Lemma}

\theoremstyle{definition}  
\newtheorem{example}[subsection]{Example}

\newtheorem{remark}[subsection]{Remark}
\newtheorem{notation}[subsection]{Notation}

  {\end{list}}

\newcommand{\dfn}{\textbf} 

\newcommand{\mdfn}[1]{\dfn{\mathversion{bold}#1}} 

\newcommand{\Smash}             {\wedge}

\newcommand{\tens}              {\otimes}               

\newcommand{\iso}               {\cong}  

\newcommand{\cat}{\EuScript}    
\newcommand{\cA}{{\cat A}}      
\newcommand{\cC}{{\cat C}}
\newcommand{\cD}{{\cat D}}

\newcommand{\cF}{{\cat F}}

\newcommand{\cH}{{\cat H}}

\newcommand{\cO}{{\cat O}}
\newcommand{\cP}{{\cat P}}

\newcommand{\GTop}{G{\cat Top}}
\newcommand{\GSpectra}{G{\cat Spectra}}
\newcommand{\Set}{{\cat Set}}
\newcommand{\GSet}{G\Set}
\newcommand{\fGSet}{G\Set_{\text{\rm fin}}}

\renewcommand{\cA}{{\mathcal A}}
\renewcommand{\cC}{{\mathcal C}}

\newcommand{\field}[1]  {\mathbb #1} 
\newcommand{\A}         {\field A}
\newcommand{\F}         {\field F}
\newcommand{\R}         {\field R}

\newcommand{\Z}         {\field Z}
\newcommand{\C}         {\field C}
\newcommand{\Q}         {\field Q}

\DeclareMathOperator*{\im}{Im}

\DeclareMathOperator{\Hom}{Hom}

\DeclareMathOperator{\End}{End}

\DeclareMathOperator{\chara}{char}

\DeclareMathOperator{\ob}{ob}
\DeclareMathOperator{\Spec}{Spec}

\newcommand{\ra}{\rightarrow}                   
\newcommand{\lra}{\longrightarrow}              
\newcommand{\la}{\leftarrow}                    
\newcommand{\lla}{\longleftarrow}               
\newcommand{\llra}[1]{\stackrel{#1}{\lra}}      
\newcommand{\llla}[1]{\stackrel{#1}{\lla}}      

\newcommand{\inc}{\hookrightarrow}              

\newcommand{\blank}{-}                          
\newcommand{\id}{id}                            

\newcommand{\adjoint}{\rightleftarrows}

\numberwithin{equation}{section}

\newcommand{\Ab}{\cat Ab}

\newcommand{\CRing}{\cat{C}omm\cat{R}ing}

\newcommand{\ev}{ev}
\newcommand{\cev}{cev}

\newcommand{\pres}[1]{\!\!\phantom{|}_{*\,}\!{#1}}

\newenvironment{myequation}
  {\addtocounter{subsection}{1}\begin{eqnarray}}
  {\end{eqnarray}$\!\!$}

\DeclareMathOperator{\GWC}{GWC}
\DeclareMathOperator{\Aut}{Aut}
\DeclareMathOperator{\Gal}{Gal}
\DeclareMathOperator{\tr}{tr}
\DeclareMathOperator{\GW}{GW}
\DeclareMathOperator{\Mackey}{{\mathcal B}{\mathit u}{\mathit
    r}{\mathit n}}
\DeclareMathOperator{\Aff}{{\cat Aff}}
\newcommand{\Or}{Or}

\newcommand{\piS}{\pi_\Theta}

\DeclareMathOperator{\fEt}{fEt}

\newcommand{\Aet}{\cA_{\fEt}}
\newcommand{\fad}{(ad)}
\DeclareMathOperator{\oalg}{alg}
\newcommand{\alg}[1]{{#1}\!-\!\oalg}
\newcommand{\fGset}{{\text{fin}G\Set}}
\newcommand{\qf}[1]{\langle{#1}\rangle}

\begin{document}

\title{Gysin functors and the Grothendieck-Witt category,
part I}

\author{Daniel Dugger}
\address{Department of Mathematics\\ University of Oregon\\ Eugene, OR
97403} 

\email{ddugger@math.uoregon.edu}

\begin{abstract}
Fix a field $k$.  Consider the motivic stable homotopy category over
$k$, and restrict to the full subcategory whose objects are the
suspension spectra of separable field extensions of $k$.  We give
an algebraic description of this category, identifying it with a
construction we call the Grothendieck-Witt category.  In this first of
two papers  we develop the general categorical machinery that
describes this situation: that of Gysin
functors and their associated categories of correspondences.  We prove
a ``recognition theorem'' for these
correspondence categories, and develop results concerning their
structure.
\end{abstract}

\maketitle

\tableofcontents

\section{Introduction}
\label{se:intro}
Fix a ground field $k$.  In this paper we describe a category
$\GWC(k)$, called the \dfn{Grothendieck-Witt category} of $k$, whose
objects are the finite separable field extensions of $k$.  The
morphisms are a Grothendieck group of certain kinds of
``correspondences'' built up from bilinear forms, and there is an
intrinsic notion of composition.  We then generalize this situation
into the theory of what we call Gysin functors and their associated categories of
correspondences.  We prove several results about the general structure of such
categories.

\medskip

To further explain the ideas and motivation of this paper we take a
brief detour into equivariant homotopy theory.  Let $G$ be a finite
group, and let $\GTop$ be the category of $G$-spaces and equivariant
maps.  We regard $\GSet$, the category of $G$-sets, as the full
subcategory of $\GTop$ consisting of the discrete $G$-spaces.  
The \dfn{orbit category} $\Or(G)$ of $G$ is the full subcategory of
$\GSet$ consisting of the $G$-sets on 
which $G$ acts transitively.  Every object in $\Or(G)$ is isomorphic
to a quotient $G/H$, for some subgroup $H$.

Next consider the stabilization functor $\Sigma^\infty\colon \GTop\ra
\GSpectra$
from $G$-spaces to genuine $G$-spectra (the version of $G$-spectra
where representation spheres are invertible).  When restricted to $\GSet$ this
map is an embedding, but it is not full.  The full subcategory of
$\GSpectra$ whose objects are $\Sigma^\infty \cO_+$ for $\cO$ a $G$-set
is called the \dfn{stable category} of $G$-sets, and
denoted $\GSet^{st}$.  We will actually focus on $\fGSet^{st}$, where
we restrict $\cO$ to be a finite $G$-set.
The full subcategory of $\fGSet^{st}$ consisting
of the objects $\Sigma^\infty(G/H)_+$ is called the \dfn{stable orbit category}.

There are two common ways of describing $\fGSet^{st}$:

\begin{enumerate}[(1)]
\item Given two finite $G$-sets $\cO_1$ and $\cO_2$, define a \dfn{span} from
$\cO_1$ to $\cO_2$ to be a diagram
\[ \xymatrix{
&P \ar[dl]\ar[dr]\\
\cO_2 && \cO_1
}
\]
in the category of finite $G$-sets.  A map between spans is a map of diagrams
that is the identity on $\cO_1$ and $\cO_2$.  This category has a monoidal
structure given by disjoint union in the ``$P$''-variable.  Define
\mdfn{$\Mackey(\cO_1,\cO_2)$} to be the Grothendieck group of isomorphisms
classes of spans from $\cO_1$ to $\cO_2$, with respect to this disjoint union
operation.

Note that we will sometimes refer to spans as ``correspondences'', as
that terminology is often used in geometric settings.  

If we have three finite $G$-sets $\cO_1$, $\cO_2$, and $\cO_3$ then we can
define a
composition of spans via the pullback operation  shown in the following diagram:
\[ \xymatrix{
&& Q\times_{\cO_2} P\ar@{.>}[dr]\ar@{.>}[dl]\\
  &Q \ar[dl]\ar[dr] && P\ar[dr]\ar[dl]\\
\cO_3 && \cO_2 && \cO_1.
}
\]
This operation induces a map
\[ \Mackey(\cO_2,\cO_3)\times \Mackey(\cO_1,\cO_2) \ra
\Mackey(\cO_1,\cO_3)
\]
which is readily checked to be unital and associative.  So we have
defined a category $\Mackey$ whose objects are the $G$-orbits.  
This is usually called the \dfn{Burnside category} of $G$-sets.

Here are some things to take note of:

\begin{enumerate}[(a)]
\item There is a functor $R\colon \fGSet\ra \Mackey$ that is the
identity on objects and sends a map $f\colon \cO_1\ra \cO_2$ to the
span $[\cO_2 \llla{f}\ \cO_1\llra{\id} \cO_1]$.

\item The category $\Mackey$ has a duality anti-automorphism $(\blank)^*$
which is the identity on objects, and on morphisms sends a span $[\cO_2\la P \ra
\cO_1]$ to the similar span $[\cO_1\la P\ra \cO_2]$ obtained by
reversing the order of the maps.  The duality functor is an
isomorphism
\[ (\blank)^*\colon \Mackey^{op}\ra \Mackey.
\]

\item In particular, setting $I=(\blank)^*\circ R$ gives a functor
$I\colon \fGSet^{op}\ra \Mackey$.  If $f\colon \cO_1\ra \cO_2$ then
$I(f)$ is the span $[\cO_1 \llla{\id} \cO_1 \llra{f} \cO_2]$.  
\end{enumerate}

\medskip

If $\cA$ is an additive category then additive functors $\Mackey^{op}\ra \cA$
are the same as what are usually called {\it Mackey functors}.  (One could also
identify Mackey functors with additive functors $\Mackey\ra \cA$, since
$\Mackey$ is self-dual; however, our notation for the $R$ and $I$ maps fits
better with the contravariant option).

It is a classical theorem (perhaps a folk theorem) 
that $\Mackey$ is isomorphic to the stable category of
finite $G$-sets.

\item The stable orbit category $\Or(G)^{st}$ can also be described in terms of
generators and relations.
 This is the free additive category whose
objects are the transitive $G$-sets and whose morphisms are generated
by the maps $R_f\colon \cO_1\ra \cO_2$ and $I_f\colon \cO_2\ra \cO_1$
for every map of  $G$-sets $f\colon \cO_1\ra \cO_2$.  The morphisms
satisfy the relations:
\begin{enumerate}[(i)]
\item $R_{gf}=R_g\circ R_f$;
\item $I_{gf}=I_f\circ I_g$;
\item Given a pullback diagram of $G$-sets
\[ \xymatrix{
 P \ar[d]_{p}\ar[r]^{f} & \cO_3 \ar[d]^{q}\\
\cO_1 \ar[r]^{g} & \cO_2
}
\]
where the actions on $\cO_1$, $\cO_2$, and $\cO_3$ are transitive,
write $P=\coprod_i X_i$ where each $X_i$ is a transitive $G$-set.
Then
\[ I_g \circ R_q = \sum_i R_{p_i}\circ I_{f_i}\]
where $f_i$ and $p_i$ are the restrictions of $f$ and $p$ to $X_i$.
\end{enumerate}
It is again a classical theorem that this category, defined in terms of generators and
relations, is isomorphic to the stable orbit category.
\end{enumerate}

Now let us return to our original setting, where
$k$ is a fixed ground field.  
Keeping the above discussion in mind, the point of this series of papers is to
examine the full subcategory of the motivic stable homotopy category
over $k$ whose objects are the suspension spectra of fields.  This is vaguely
analogous to the stable orbit category (although in the case of
$G$-spectra the orbits generate the category, whereas field spectra do
not generate the category in the motivic setting).  
Our goal is
to give descriptions of this category that are analogs of
(1) and (2).  To give a sense of this in the first case, the
Grothendieck-Witt category of $k$ is defined to be the category
$\GWC(k)$ whose objects are $\Spec E$ for $E$ a finite, separable
field extension of $k$.  The morphisms from $\Spec E$ to $\Spec F$ are
the Grothendieck group $GW(F\tens_k E)$ of quadratic spaces over
$F\tens_k E$ (see Section~\ref{se:GWcat} for details).  The definition of
composition is a little too cumbersome to be included in this
introduction, but it mimics the composition we saw in (1) above.

\medskip

Morel \cite{Mo}  proved that if $k$ is perfect and 
$F/k$ is a separable field extension then
\[ [\Sigma^\infty (\Spec F)_+,S]\iso \GW(F)
\]
where $S$ is the motivic sphere spectrum and $[\blank,\blank]$ denotes
maps in the motivic stable homotopy category of smooth $k$-schemes.  
If $J/k$ is another
separable extension one can then argue that
\begin{align*}
 [\Sigma^\infty (\Spec F)_+,\Sigma^\infty(\Spec J)_+]&\iso
[\Sigma^\infty (\Spec F)_+\Smash \Sigma^\infty(\Spec J)_+,S]\\
& \iso
[\Sigma^\infty (\Spec (F\tens_k J))_+,S] \\
&\iso \GW(F\tens_k J)
\end{align*}
where the first isomorphism uses a self-duality $\Sigma^\infty(\Spec
J)_+\iso \cF(\Sigma^\infty(\Spec J)_+,S)$ and the last isomorphism is
the aforementioned one of Morel (using that $F\tens_k J$ decomposes as
a product of separable field extensions of $k$).  The self-duality is
dealt with in the appendix to \cite{H}, and in the equivariant context
it is in modern times usually couched in the machinery of the Wirthm\"uller
isomorphism (cf. \cite{Ma2}, for example).

Accepting the above computation, it remains to compute the composition
in the motivic stable homotopy category and relate it to the
appropriate pairing of Grothendieck-Witt groups.  The present paper
exists partly because attempting to do this by ad hoc methods
proved unwieldy.

In the narrative we provide here, everything comes
down to the existence of transfer maps.  Transfer maps coupled with
diagonal maps give rise to duality structures, and quite general
categorical computations show that any reasonable category with this
kind of structure may be described by a ``correspondence-like''
description of composition.

\medskip 

Let us now explain the results in a bit more detail.
Let $\cC$ be a finitary lextensive category (see
Section~\ref{se:Gysin-func}, but understand
that this is basically just a category where coproducts behave nicely
with respect to pullbacks).  A \dfn{Gysin functor} on $\cC$ is an
assignment $X\mapsto E(X)$ from $\ob(\cC)$ to commutative rings,
together with pullback and pushforward maps satisfying certain
compatibility properties.  Given this situation, one can 
construct a \dfn{category of correspondences} $\cC_E$ where the object
set is $\ob(\cC)$, maps from $X$ to $Y$ are the abelian group
$E(Y\times X)$, and composition is obtained by a familiar formula
using the pullback and pushforward maps.  The category $\cC_E$ is
enriched over abelian groups, is closed symmetric monoidal, and has
the property that all objects are self-dual.

Now suppose $\cH$ 
is a closed tensor category (additive category with compatible
symmetric monoidal structure) with tensor $\tens$ and unit
$S$.  
Suppose given functors $R\colon \cC\ra \cH$ and $I\colon \cC^{op}\ra
\cH$ satisfying some reasonable hypotheses (see Section~\ref{se:recon}).  For
$f\colon X\ra Y$ in $\cC$ we think of $Rf$ as the ``regular'' map
associated to $f$ in $\cH$, whereas $If$ is an associated transfer map.
The
prototype for this situation is where $\cH$ is the genuine
$G$-equivariant
stable homotopy
category, $\cC$ is the category of finite $G$-sets,
$R(X)=I(X)=\Sigma^\infty(X_+)$,  $Rf$ is the usual map induced by
$f\colon X\ra Y$, and $If$ is the corresponding transfer map.  

Write $\pi^0$ for the functor $\cC^{op}\ra \Ab$ 
given by $\pi^0(X)=\cH(RX,S)$.  This inherits the structure of a Gysin
functor, and we prove the following:

\begin{thm}
Under mild hypotheses, the category of
correspondences $\cC_{(\pi^0)}$ is equivalent to the full subcategory
of $\cH$ whose objects lie in the image of $R$.
\end{thm}

That is, we prove that one can reconstruct the appropriate subcategory
of $\cH$ as the category of correspondences associated to the Gysin
functor $\pi^0$.  See Theorem~\ref{th:recon} for a precise version of
the above theorem.

\bigskip

The second result of this paper concerns the structure of the category
of correspondences $\cC_E$ for a general Gysin functor $E$.  In the
Burnside category of a finite group, there are special collections of
maps $Rf$ and $Ig$ and every map in the category may be written as a
composite $Rf\circ Ig$.  There are also rules for rewriting
compositions $If\circ Rg$ in the above form.  In the case of a general
Gysin functor, there are {\it three\/} collections of special maps, elements
of which are written $Rf$,
$Ig$, and $Da$ where $f$ and $g$ are maps in $\cC$ and $a\in E(X)$ for
some object $X$ in $\cC$.  We prove the following:

\begin{thm}
\label{th:intro-RDI}
Every map in $\cC_E$ can be written as a sum of maps $Rf\circ
Da\circ Ig$.  Other composites of the $R-D-I$ maps can be
rewritten in this
form 
using the rules
\begin{enumerate}[(a)]
\item $Da\circ Rf=Rf\circ D(f^*a)$,
\item $If\circ Da=D(f^*a)\circ If$,
\item $If\circ Rg=Rp\circ Iq$ where $p$ and $q$ are the maps in the
pullback diagram
\[ \xymatrix{
P \ar[r]^q\ar[d]_p & A \ar[d]^g \\
B \ar[r]^f & C
}
\]
inside the category $\cC$.  
\end{enumerate}
Moreover, for a map $f\colon X\ra Y$ in $\cC$ and $a\in E(X)$ one has
$Rf\circ Da\circ If=D(f_!(a))$.  
\end{thm}

The following corollary is really just a reformulation of the theorem:

\begin{cor}
\label{co:intro-RDI}
Maps in $\cC_E(X,Y)$ can be represented by a
pair consisting of a span $[Y\llla{f} Z \llra{g} X]$ and an element $a\in
E(Z)$: this pair represents $Rf\circ Da\circ Ig$.  If a map in 
$\cC_E(U,X)$ is
represented by $[X\llla{f'} Z' \llra{g'} U,\ a'\in E(Z')]$ then the composite
is represented by the pullback span  
\[ \xymatrixcolsep{1pc}\xymatrixrowsep{1pc}\xymatrix{
&& P\ar[dl]_s\ar[dr]^t \\
& Z\ar[dl]_f \ar[dr]^g && Z'\ar[dl]_{f'}\ar[dr]^{g'} \\
Y && X && U.
}
\]
and the element $D\bigl ((s^*a)(t^*a')\bigr)\in E(P)$.  That is to
say,
\[ (Rf\circ Da\circ Ig)\circ (Rf'\circ Da'\circ Ig')=R(fs)\circ
D((s^*a)(t^*a'))\circ I(g't).\]  
Moreover, we have the extra relation
\[ 
\biggl [
Y \llla{f} Z \llra{f} Y, \ a\in E(Z) \biggr ]
=
\biggl [
Y \llla{\id} Y \llra{\id} Y, \ f_!(a)\in E(Y) \biggr ].
\]
\end{cor}

If one assumes the category $\cC$ to have
some basic Galois-type properties (which model the behavior of the
category of $G$-sets) then explicit computations become easier.  For example,
one can prove the following:

\begin{prop}
Assume $\cC$ is a Galoisien category (see
Section~\ref{se:Gysin-Galois}), 
and let $X$ be
an object in $\cC$ that is Galois.  Then in $\cC_E$ one has
\[ \End_{\cC_E}(X)=
\widetilde{[\Aut_\cC(X)]}E(X)
\]
where on the right we have the twisted group ring whose elements are
finite sums $\sum_i [g_i]a_i$ with $g_i\in \Aut_\cC(X)$ and $a_i\in
E(X)$, and the multiplication is determined by the formula
\[ [g]a\cdot [h]b = [gh](h^*a\cdot b).
\]
(Here  $[g]a$ corresponds to the element $Rg\circ Da$).
\end{prop}

The above proposition describes the full subcategory of $\cC_E$
consisting of a single Galois object.  In a similar vein, one can 
explicitly describe the full subcategories generated by multiple
Galois objects.  See
Section~\ref{se:struc}.

\medskip

Although the motivation for this paper comes from a concrete question concerning
motivic homotopy theory, here we only develop the
categorical backdrop.  In a sequel \cite{D2} we will explain how this backdrop
applies to both the $G$-equivariant and motivic settings.

\subsection{Organization of the paper}
In Section 2 we write down a complete definition of the
Grothendieck-Witt category.  In Section 3 we generalize this, by
introducing the notions of a Gysin functor 
and its associated category of correspondences (a Gysin functor is
the same thing as what is called a commutative Green functor in the
group theory literature).  Section 4 continues
the development of this machinery and proves the main ``reconstruction
theorem'' (which in this generality is a simple exercise in category theory).  

Section 5 gives a deeper investigation into the structure of
correspondence categories, and serves as a prelude to Section 6 where
we work out some basic computations inside 
Grothendieck-Witt categories over a field.  

\subsection{Notation and terminology}

The common notation ``$f(x)$'' establishes a right-to-left trend in
symbology: one starts with $x$ and then applies $f$ to it.  The common
notation $\Hom(A,B)$ is based on the opposite left-to-right trend.
The opposing nature of these two notations is one of the most common
annoyances in modern mathematics.  Our general philosophy in this
paper is that we will always use the right-to-left convention, except
when we write $\Hom(A,B)$.  This has already appeared in our
treatement of the Burnside category, where spans from $\cO_1$ to
$\cO_2$ were drawn with the $\cO_1$ term on the right.  That
particular convention will have various incarnations throughout the
paper.

The projection map $X\times Y\times Z \ra X\times Z$ will be written
$\pi^{XYZ}_{XZ}$, and similarly for other projection maps.  If
$f\colon A\ra X$ and $g\colon A\ra Y$, then it is sometimes useful to
denote the induced map $A\ra X\times Y$ as $f\times g$.  
Unfortunately, $f\times g$ also denotes the map $A\times A\ra X\times
Y$.  Usually it is clear from context which one is meant, but when
necessary we will write $(f\times g)^A_{XY}$ and $(f\times
g)^{AA}_{XY}$ to distinguish them.  In all these conventions, the
superscript is the domain and the subscript is the range.

\subsection{Acknowledgments}  
I am grateful to Jeremiah Heller and Kyle Ormbsy for expressing
interest in these results, for their diligence in tracking down the
reference \cite{Dr}, and for useful conversations.  Likewise, I am
grateful to Ang\'elica Osorno for a very helpful and inspiring
discussion.


\section{Background on Grothendieck-Witt groups and composition}
\label{se:GWcat}

In this section we recall the definition and basic properties of the
Grothendieck-Witt group.  Then we explain how these groups can be
assembled to give the hom-sets in a certain category.  

\medskip

\subsection{Grothendieck-Witt groups}
\label{se:GW-groups}
Let $R$ be a commutative ring.  A \dfn{quadratic space} over $R$ is a
pair $(P,b)$ consisting of a finitely-generated, projective $R$-module
$P$ together with a map $b\colon P\tens_R P \ra R$ that is symmetric
in the sense that $b(x,y)=b(y,x)$ for all $x,y\in P$.  One says that
$(P,b)$ is \dfn{nondegenerate} if the adjoint map $P\ra \Hom_R(P,R)$
associated to  $b$ is an 
isomorphism of $R$-modules.

Given any
maximal ideal $m$ of $R$ there is an induced map
\[ \xymatrix{
P\tens_R P \ar[r]^b\ar[d] & R\ar[d] \\
(P/mP) \tens_{R/m} (P/mP) \ar@{.>}[r]^-{b_m} & R/m
}
\]
giving a symmetric bilinear form $b_m$ on the $R/m$-vector space
$P/mP$.  One readily checks that $(P,b)$ is nondegenerate if and only
if $(P/mP,b_m)$
is nondegenerate for every maximal ideal $m$ of $R$.  In many cases
nondegeneracy is most easily checked using this criterion.

It will be useful for us to sometimes think geometrically.  A
quadratic space is an algebraic vector bundle on $\Spec R$ equipped
with a fibrewise symmetric bilinear form, and it is nondegenerate if
the bilinear forms on the closed fibers are all nondegenerate.
 
Note that there is an evident direct sum operation on quadratic
spaces.  There is also a tensor product: if $(P,b)$ and $(Q,c)$ are
quadratic spaces then $(P\tens_R Q,b\tens_R c)$ denotes the projective
module $P\tens_R Q$ equipped with the bilinear form
\[ \xymatrixcolsep{2.5pc}\xymatrix{
(P\tens_R Q)\tens_R (P\tens_R Q) \ar[r]^-{\id\tens t\tens \id}
& P\tens_R P\tens_R Q\tens_R Q \ar[r]^-{b\tens c} & R\tens_R R
\ar[r]^-{\mu} & R.
}
\]
It is easy to check that the direct sum and tensor product of
nondegenerate quadratic spaces are again nondegenerate.

The \dfn{Grothendieck-Witt group} of $R$, denoted $\GW(R)$,
is the Grothendieck group of nondegenerate quadratic spaces with
respect to direct sum.  It has a ring structure induced by tensor
product.
If $f\colon R\ra S$ is a map of commutative rings then there is an
induced map of rings
$f_*\colon \GW(R)\ra \GW(S)$ given by $(P,b)\mapsto (P\tens_R
S,b\tens_R \id_S)$.

It turns out that $\GW(\blank)$ is also a contravariant functor, but
only with respect to certain kinds of maps.  We explain this next.

\begin{defn}
A map of commutative rings $R\ra S$ is \dfn{sheer} if $S$ is a
finitely-generated, projective $S$-module.  
\end{defn}

When $R\ra S$ is sheer there is a trace map $\tr_{S/R}\colon S\ra R$
defined in the evident way: $\tr_{S/R}(s)$ is the trace of the
multiplication-by-$s$ map $x\mapsto xs$ on $S$.  The map
$\tr_{S/R}$ 
is $R$-linear.

If $f\colon R\ra S$ is sheer  then there is a map
$f^!\colon \GW(S)\ra \GW(R)$ defined as follows: if $(Q,c)$ is a quadratic space over
$S$ then we let $f^!(Q,c)$ be $Q$ regarded as an $R$-module (via
restriction of scalars along $f$) equipped with the bilinear pairing
\[ \xymatrix{
Q\tens_R Q \ar[r] & Q\tens_S Q \ar[r]^-c &  S \ar[r]^{\tr_{S/R}} &R.
}
\]
Note that $f^!$ will usually not be a map of
rings.  

\begin{remark}
The above material on the Grothendieck-Witt group is standard, and can be
found in \cite{S}.  The map $f_!$ is sometimes called the
\dfn{Scharlau transfer}; one can find it in \cite[Chapter 2.5]{S} 
\end{remark}

\subsection{Separable algebras} The following material is classical,
but perhaps not as readily accessible in the literature as
it could be.  See \cite{J}, \cite{DI}, and \cite{L}, though.

\begin{defn}
  Let $A\ra B$ be a map of commutative rings.  We say that $B$ is a
  \dfn{separable} $A$-algebra if any of the following equivalent conditions is
  satisfied:
\begin{enumerate}[(1)]
\item $B$ is projective as a $B\tens_A B$-module,
\item The multiplication map $\mu\colon B\tens_A B \ra B$ is split in
the category of $B\tens_A B$-modules,
\item There exists an element $\omega\in B\tens_A B$ such that
$\mu(\omega)=1$ and $(b\tens 1)\omega=\omega(1\tens b)$ for all $b\in
B$.
\item The ring map $B\tens_A B\ra B$ is sheer.
\end{enumerate}
\end{defn}

The equivalence of the conditions in the above definition is straightforward:
clearly (1)$\Leftrightarrow$(2), and (2)$\Leftrightarrow$(3) by letting $\omega$
be the image of $1$ under the splitting.  Note that $B\tens_A B\ra B$ is
necessarily surjective, and so $B$ is always cyclic as a $B\tens_A B$-module
(and in particular, finitely-generated).  This explains why (1) is equivalent to
(4).

If $\omega$ is a class as in (3) of the above definition, then for any
$z\in B\tens_A B$ one has
\[ z. w = (\mu(z)\tens 1).w = w.(1\tens \mu(z)).
\]
To see this, write $z=\sum a_i\tens b_i$ and then just compute that
\begin{align*}
 z.w=\sum (a_i\tens b_i).w = \sum (a_i\tens 1)(1\tens b_i).w &=\sum
(a_i\tens 1)(b_i\tens 1).w \\
&= \Bigl ( \bigl (\sum a_ib_i\bigr )\tens 1\Bigr
).w=(\mu(z)\tens 1).w.
\end{align*}
In particular, notice that the class $\omega$ from (3) will be unique:
if $\omega'$ is another such class then we would have
\[ \omega.\omega'=(\mu(\omega)\tens 1).\omega' = (1\tens
1).\omega'=\omega'
\]
and likewise $\omega.\omega'=\omega$.  Also notice that $\omega$ is
idempotent.  Consequently, we have the isomorphism of rings
\[ B\tens_A B\iso (B\tens_A B)/w \,\times\, (B\tens_A B)/(1-w)
\]
given in each component by projection.  The second component can be
identified with $B$.  Indeed, certainly $1-\omega$ belongs to $\ker
\mu$.  Conversely, if $s\in \ker\mu$ then $s.\omega=(\mu(s)\tens
1).\omega=0$, and so $s=s-s.\omega=(1-\omega)s \in (1-\omega)$.  It
follows that $\mu$ induces an isomorphism of rings $(B\tens_A
B)/(1-\omega) \iso  B$.

\begin{remark} 
It helps to have some geometric intuition here.  When $E\ra B$ is a
topological covering space, the diagonal $\Delta\colon E\ra E\times_B
E$ gives a homeomorphism from $E$ onto a particular component of
$E\times_B E$.  Similarly, when $A\ra B$ is separable then $\Spec
B\times_{\Spec A} \Spec B$ splits off the diagonal copy of $\Spec B$
as one connected component.  The idempotent $\omega\in B\tens_A B$ is
the algebraic culprit for this splitting.
\end{remark}

\begin{remark}
There is another description of $\omega$ that is sometimes useful.
Since $B$ is a finitely-generated, projective $B\tens_A B$-module
there is a trace map $\tr_{B/(B\tens_A B)}\colon \End(B)\ra B\tens_A B$.
The element $\omega$ is simply $\tr_{B/(B\tens_A B)}(\id_B)$.
\end{remark}

When $B$ is a separable $A$-algebra there is a canonical quadratic
space over the ring $B\tens_A B$: it is $B$ itself (with the usual structure of
$B\tens_A B$-module), equipped with the following bilinear
form:
\[ B\tens_{(B\tens_A B)} B \ra B\tens_A B, \qquad x\tens y \ra (xy\tens 1).\omega.
\]
A moment's check shows that this is indeed $B\tens_A B$-bilinear, as
required.      We will denote this quadratic space as
$(B,\mu\cdot \omega)$.  

More generally, for any quadratic space $(P,b)$ over $B$ we obtain a
quadratic space $(P,b\cdot\omega)$ over $B\tens_A B$.  The underlying
module is $P$ (regarded as a $B\tens_A B$-module, where it is
necessarily projective) equipped with the
bilinear form
\[ P\tens_{(B\tens_A B)} P \ra B\tens_A B, \qquad
x\tens y\mapsto (b(x,y)\tens 1).\omega.
\]
This construction induces a map of groups (not rings)
\[ \GW(B) \ra \GW(B\tens_A B), \qquad [P,b]\mapsto [P,b\cdot \omega].
\]
Of course this is just the map $\mu^!$ defined at the end of
Section~\ref{se:GW-groups}, where $\mu$ is the multiplication
$B\tens_A B\ra B$.

\begin{remark}
The significance of the quadratic space $(B,\mu\cdot \omega)$ will
become clear in Section~\ref{se:GWsub} below.  It plays the role of
the identity morphism in the Grothendieck-Witt category.  
\end{remark}

\medskip

We will shortly restrict ourselves to studying maps $R\ra S$ which are
both sheer and separable.  Such maps are commonly referred to by
another name:

\begin{prop}
\label{pr:Noeth-ssep}
Assume that $R$ is Noetherian.  Then $R\ra S$ is both sheer and
separable if and only if $R\ra S$ is finite and \'etale.  
\end{prop}

\begin{proof}
Suppose $R\ra S$ is sheerly separable.  Then $R\ra S$ is automatically
finite and flat.  
Consider the exact sequence
\[ 0 \ra I \lra S\tens_R S \llra{\mu} S \ra 0.
\]
Since $R\ra S$ is separable, this is split as a sequence of $S\tens_R
S$-modules.  So there is an $S\tens_R S$-linear map $\chi\colon S\tens_R S\ra I$
splitting the inclusion.  Linearity implies that this map sends $I$
into $I^2$, and so surjectivity gives us $I=I^2$.  So
$\Omega_{S/R}=I/I^2=0$.  Since $S$ is flat and finite-type over $R$,
and $\Omega_{S/R}=0$, it follows that $R\ra S$ is \'etale by
\cite[Proposition I.3.5]{Mi}.

Now suppose that $R\ra S$ is finite \'etale.  Since $R\ra S$ is flat,
$R$ is Noetherian, and $S$ is finitely-generated, it follows from
\cite[Corollary 6.6]{E} that $S$ is projective over $R$.  So $R\ra S$
is sheer.

The map $f\colon S\ra S\tens_R S$ given by $f(s)=s\tens 1$ is also
\'etale (geometrically, \'etale maps are closed under pullback).  If
$\mu\colon S\tens_R S\ra S$ is the multiplication, then $\mu \circ f=\id$.
Since $f$ and $\id$ are \'etale, so is $\mu$ by \cite[Corollary
I.3.6]{Mi}.  Therefore $S$ is flat over $S\tens_R S$.  But $S$ is
finite-type over $R$ and $R$ is Noetherian, hence $S$ and $S\tens_R S$
are both Noetherian as well.  Since $S$ is both flat and
finitely-generated (in fact, cyclic) over $S\tens_R S$,
it is actually
projective by \cite[Corollary 6.6]{E} again.  So $R\ra S$ is separable.
\end{proof}

\begin{cor}
Let $k$ be a field.  A map of commutative rings $k\ra E$ is sheer and
separable if and only if there is an isomorphism of $k$-algebras
$E\iso E_1\times E_2\times\cdots \times E_n$ where each $E_i$ is a
separable (in the classical sense) field extension of $k$.
\end{cor}

\begin{proof}
By Proposition~\ref{pr:Noeth-ssep} we can replace ``sheerly separable'' by
``finite \'etale'', and then the result is standard (for example, see
\cite[Proposition I.3.1]{Mi}).
\end{proof}

\begin{remark}
Suppose we are working in a category that has finite limits.  Let
$\cP$ be a property of morphisms that is closed under composition and
pullback.  Say that a morphism $X\ra Y$ has property $\cP\cP$ if $X\ra
Y$ has $\cP$ and $\Delta\colon X\ra X\times_Y X$ also has $\cP$.  Then
it follows by general category theory that property $\cP\cP$ is closed
under composition and pullback, and has the feature that if composable
morphisms $X\llra{f} Y \llra{g} Z$ are given such that both $f$ and
$gf$ have $\cP\cP$ then so does $g$.  For the proof of the latter, the
main ideas can be found in any standard reference dealing with the
case where $\cP$ is ``\'etale'' (e.g. \cite[Corollary I.3.6]{Mi}).
In the present context, we can apply this principle to the opposite
category of commutative rings, where $\cP$ is ``sheer'' and $\cP\cP$
is therefore ``sheerly separable''.   So the sheerly separable maps
are closed under pullbacks and composition, and have the indicated
two-out-of-three property.  
\end{remark}

\begin{example}
Here are three examples to keep in mind when dealing with these
concepts:
\begin{enumerate}[(a)]
\item If $R$ and $S$ are commutative rings then the projection 
$R\times S\ra R$ is sheerly separable, but not an injection.
\item If $R$ is a commutative ring then the map $R[x]\ra R$ sending
$x\mapsto 0$ is separable but not sheer.
\item Given any non-separable, finite field extension $k\inc E$,
this map is sheer but not separable.  
\end{enumerate}
\end{example}

\begin{remark}
The maps we are calling ``sheerly separable'' are called ``strongly
separable'' in \cite{J}, and ``projective separable'' in \cite{L}.
The following two conditions on a map of commutative rings $R\ra S$ are
also equivalent to being sheerly separable:
\begin{enumerate}[(1)]
\item $S$ is separable over $R$ and $S$ is projective as an $R$-module
  (but not required to be finitely-generated);
\item $S$ is a finitely-generated projective module over $R$ and the
trace form $S\tens_R S\ra R$ (given by $x\tens y\mapsto \tr_{S/R}(xy)$) is nondegenerate.   
\end{enumerate}
The proof of these equivalences, or at least a sketch of such, is
available in \cite[Proposition 6.11]{L}.  We will not need either of
these characterizations in the present paper.
\end{remark}

\subsection{The Grothendieck-Witt category of a commutative ring}
\label{se:GWsub}
We next restrict to a somewhat specialized setting.  Assume that $S$,
$T$, and $U$ are $R$-algebras, but also assume that $R\ra T$ is sheer
and separable.

Now suppose given a quadratic space $(Q,c)$ over $U\tens_R T$ and
another quadratic space $(P,b)$ over $T\tens_R S$.  In the following
diagram, it is readily checked that the ``across-the-top-then-down''
composite satisfies the appropriate $T$-invariance condition to induce
the dotted map:
\[ \xymatrix{
Q\tens_R P\tens_R Q \tens_R P \ar[r]^{1\tens t\tens 1} \ar[ddd] & 
Q\tens_R Q\tens_R P\tens_R P
\ar[r]^-{c\tens b} & (U\tens_R T)\tens_R (T\tens_R S) \ar[d]^{1\tens
\mu\tens 1} \\
&& U\tens_R T\tens_R S \ar[d]^{1\tens \tr_{T/R}\tens 1} \\
&& U\tens_R R\tens_R S\ar[d]^\iso \\
(Q\tens_T P)\tens_R(Q\tens_T P) \ar@{.>}[rr]^-{c\hat{\tens}_T b} && U\tens_R S.
}
\]
This produces a quadratic space $(Q\tens_T P,c\hat{\tens}_T b)$ over
the ring $U\tens_R S$.  It is easy to check that this is nondegenerate
if $(P,b)$ and $(Q,c)$ were, and the construction is evidently
compatible with direct sums.   So we obtain a pairing
\begin{myequation}
\label{eq:GWC-comp}
\GW(U\tens_R T)\tens \GW(T\tens_R S) \llra{\hat{\tens}_T} \GW(U\tens_R S).
\end{myequation}
It is easy to check that these pairings satisfy associativity.  They
are also unital, with the unit being the canonical element 
$(T,\mu\cdot\omega)$ in $\GW(T\tens_R T)$.  

If we denote the evident maps as
\[ j_{12}\colon U\tens_R T \ra U\tens_R T\tens_R S, \quad
j_{23}\colon T\tens_R S \ra U\tens_R T\tens_R S,
\]
\[
j_{13}\colon U\tens_R S \ra U\tens_R T\tens_R S
\]
then the pairing of (\ref{eq:GWC-comp}) can also be expressed as
\[ \alpha\hat{\tens}_T \beta = j_{13}^!\bigl (
({j_{12}}_*\alpha) \cdot ({j_{23}}_*\beta) \bigr )
\]

\begin{defn}
Let $R$ be a commutative ring.  The \dfn{Grothendieck-Witt category}
of $R$ is the category enriched over abelian groups defined as follows:
\begin{enumerate}[(1)]
\item The objects are $\Spec T$ for $T$ a sheerly separable
$R$-algebra,
\item The set of morphisms from $\Spec T$ to $\Spec U$ is the additive
group $\GW(U\tens_R T)$;
\item Composition of morphisms is defined by (\ref{eq:GWC-comp}).
\end{enumerate}
This category will be denoted $\GWC(R)$.  
\end{defn}


\section{The general theory of Gysin functors}
\label{se:Gysin}

When studying the Grothendieck-Witt categories
$\GWC(R)$, it turns out to be advantageous to investigate the story in
greater generality.  We do this in the present section.  The
``Gysin functors'' that we introduce here are simply functors with
pullback and pushforward maps which are compatible in familiar ways.
Certainly such functors have been encountered time and again in the
literature, and so it is unlikely that anything in this section is
actually ``new''.  A very early reference is \cite{G}, whereas a more
recent reference is \cite{B}.  In the setting of finite group theory,
our Gysin functors are precisely the {\it commutative
  Green functors}.  

Being unaware of a reference that
serves as a perfect source for what we need, we take some time here to
develop the theory from first principles.  In doing so, we have tried
to provide a unity of discussion that justifies this.  We stress,
though, that much of the material from this section is in \cite{B}.

The main things we do here are:
\begin{itemize}
\item Give the definition of a Gysin functor and develop the basic properties;
\item Observe the existence of a ``universal'' Gysin functor, called
the Burnside functor;
\item Observe that any Gysin functor $E$ on a category $\cC$ 
gives rise to an associated closed, symmetric monoidal category, denoted $\cC_E$, of
``$E$-correspondences'' between the objects of $\cC$.    These
symmetric monoidal categories have the properties that all objects are
dualizable, and moreover every object is self-dual.
\end{itemize}

\medskip

\subsection{Gysin functors}
\label{se:Gysin-func}
Let $\cC$ be a category with finite limits and finite coproducts,
with the property that pullbacks distribute over coproducts: that is,
given any maps $A\ra X$, $P_1\ra X$, and $P_2\ra X$ the natural map
\[ (A\times_X P_1)\amalg (A\times_X P_2) \ra A\times_X (P_1\amalg P_2)
\]
is an isomorphism.  We also assume that for any objects $A$ and $B$ in
$\cC$ the following diagrams are pullbacks:
\[ \xymatrix{
A \ar[r]^{\id}\ar[d]_{\id} & A\ar[d]^{i_0} & B\ar[d]_{\id}\ar[r]^{\id} & B\ar[d]_{i_1} &
\emptyset \ar[d] \ar[r] & B\ar[d]_-{i_1} \\
A \ar[r]^-{i_0} & A\amalg B & B \ar[r]^-{i_1} & A\amalg B & A
\ar[r]^-{i_0} & A \amalg B.
}
\]
Such categories are called \dfn{finitary lextensive} \cite[Corollary 4.9]{CLW}.  
Standard examples to keep in mind are the
categories $\Set$ and $\GSet$ (and more generally, any topos).

\begin{defn}
\label{de:Gysin}
A \dfn{Gysin functor} on $\cC$ is a
a contravariant functor $E$ from $\cC$ to $\CRing$ together
with a covariant functor $\tilde{E}\colon \cC\ra \Ab$ such
that $E(X)=\tilde{E}(X)$ for every object $X$.  If $f\colon X\ra Y$ is
a map we write $f^*=E(f)$ and $f_!=\tilde{E}(f)$.  The maps $f_!$ will
be called Gysin maps.    For $a\in E(X)$ and $b\in E(Y)$ we write
\[ a\tens b=(\pi^{XY}_X)^*(a)\cdot (\pi^{XY}_Y)^*(b).
\]
We require the
following  axioms:
\begin{enumerate}[(1)]
\item {} {\rm [Zero axiom]} $E(\emptyset)=0$.  
\item {} {\rm [Behavior on sums]} For any objects $X$ and $Y$, the natural map
\[ i_X^*\times i_Y^*\colon E(X\amalg Y) \ra E(X) \times E(Y) \]
is an isomorphism of rings.
Here $i_X\colon X\ra X\amalg Y$
and $i_Y\colon Y\ra X\amalg Y$ are the canonical maps.
\item {}{\rm [Push-product axiom]}
For any maps
$f\colon X\ra X'$, $g\colon Y\ra Y'$ and $a\in E(X)$, $b\in E(Y)$ one
has
\[
(f\times g)_!(a\tens b)=f_!(a)\tens
g_!(b).
\]

\item {}{\rm [Push-Pull axiom]} For every pullback diagram
\[ \xymatrix{
A \ar[d]_p \ar[r]^f & B\ar[d]^q \\
C \ar[r]^g & D
}
\]
one has $f_! p^*=q^*g_!$.
\end{enumerate}
A natural transformation between Gysin functors is a natural
transformation of contravariant functors that is also a natural
transformation of the covariant piece.
\end{defn}

\begin{remark}
\mbox{}\par
\begin{enumerate}[(a)]
\item
The above definition starts with the ``internal'' multiplications on
the abelian groups $E(X)$ and derives the external pairings
$E(X)\tens E(Y)\ra E(X\times Y)$.  As usual, the opposite approach can
also be taken: we could have written the above definition in terms of
external pairings, and then constructed the internal pairings using the
diagonal maps.  The two approaches are clearly equivalent.
\item When $\cC$ is the category of finite $G$-sets, what we have
called Gysin functors are more commonly called {\it commutative Green
functors\/}; see \cite[Chapter 2]{B}.   We adopted the term 
``Gysin functor'' due to its brevity.
\end{enumerate}
\end{remark}

The following lemmas are useful to record:

\begin{lemma}
\label{le:iso-shriek}
If $f$ is an isomorphism in $\cC$ then $f_!=(f^*)^{-1}$ in any Gysin functor.
\end{lemma}

\begin{proof}
This follows immediately from the push-pull formula, using the
pullback diagram
\[ \xymatrix{
A \ar[r]^{\id}\ar[d]_{\id} & A \ar[d]^f \\
A \ar[r]^f & A.
}
\]
\end{proof}

\begin{lemma}
\label{le:Gysin-sum}
For any objects $A$ and $B$, the composition
\[ \xymatrixcolsep{3.5pc}\xymatrix{
E(A)\oplus E(B) \ar[r]^-{(i_0)_!\oplus (i_1)_!} & E(A\amalg B) \ar[r]
^-{(i_0)^*\times (i_1)^*} & E(A)\times E(B)
}
\]
sends a pair $(x,y)$ to $(x,y)$ (we refrain from calling this the identity only
because the domain and target are perhaps not ``equal'').
Consequently, the pushforward map $(i_0)_!\oplus (i_1)_!\colon E(A)\oplus E(B)\ra
E(A\amalg B)$ is an isomorphism of abelian groups.
\end{lemma}

\begin{proof}
Left to the reader.  For the first statement use the push-pull axiom
applied to the three pullback squares listed in the original
introduction of $\cC$, together with $E(\emptyset)=0$.  The second
statement of the lemma then follows directly from Axiom (2) in the
definition of Gysin functor.
\end{proof}

\begin{lemma}
\label{le:Gysin-product}
Let $f\colon A\ra X$ and $g\colon B\ra X$.  Then $(f\times_X
g)_!(1)=f_!(1)\cdot g_!(1)$.  
\end{lemma}

\begin{proof}
Use push-pull for the square
\[ \xymatrix{
A\times_X B \ar[r]\ar[d]_{f\times_{\!X} g} & A\times B \ar[d]^{f\times g}
\\
X \ar[r]^-{\Delta} & X\times X.
}
\]
Start with $1\tens 1\in E(A\times B)$, and use the push-product axiom.
\end{proof}

\begin{prop}[Projection formula]
Let $E$ be a Gysin functor.  Then given $f\colon X\ra Y$, $\alpha\in E(X)$, and $\beta\in E(Y)$ one has
\[ f_!(\alpha\cdot f^*(\beta))=f_!(\alpha).\beta.
\]
\end{prop}

\begin{proof}
Using push-pull applied to $\alpha\tens \beta\in E(X\times Y)$,
the pullback diagram
\[ \xymatrix{
X\ar[r]^-{\id\times f}\ar[d]_f & X\times Y \ar[d]^{f\times \id} \\
Y\ar[r]^-\Delta & Y\times Y
}
\]
implies that $f_!(\alpha\cdot f^*\beta)=\Delta^*(f\times
\id)_!(\alpha\tens\beta)$.  The push-product axiom finishes the proof.
\end{proof}

\begin{example}
\label{ex:Gysin}
\mbox{}\par
\begin{enumerate}[(a)]
\item Let $\cC$ be the category of sets but with morphisms the maps where
all fibers are finite (called \dfn{quasi-finite maps} from now on).  
Let $E(S)=\Hom(S,\Z)$, with the ring operations given by pointwise
addition and multiplication.  If $f\colon S\ra T$ then $f^*$ is the
evident map and $f_!\colon
E(S)\ra E(T)$ sends a map $\alpha \colon S\ra \Z$ to the assignment
$t\mapsto \sum_{s\in f^{-1}(t)} \alpha(s)$.

\item Let $G$ be a finite group, and let $\cC$ be the category of
finite $G$-sets. For $S$ in $\cC$ define $\cA(S)$ to be the
Grothendieck group of maps $X\ra S$ (where $X$ is a finite $G$-set),
made into a ring via $[X\ra S]\cdot [Y\ra S]=[X\times_S Y\ra S]$.
Given $f\colon S\ra T$ one gets maps $f^*\colon \cA(T)\ra \cA(S)$  by
pulling back along $f$, and $f_!\colon \cA(S)\ra \cA(T)$ by
composing with $f$. 
\item Let $\Aff_{sh}$ be the opposite category of commutative rings and
sheer maps.  For $R$ a commutative ring we write $\Spec R$ for the
corresponding object of $\Aff$.  
Let $K^0(\Spec R)$ be the Grothendieck group of
finitely-generated $R$-projectives.  For $f\colon \Spec R\ra \Spec S$ we have
$f^*\colon K^0(\Spec S)\ra K^0(\Spec R)$ given by $[P]\mapsto
[P\tens_S R]$, and $f_!\colon K^0(\Spec R)\ra K^0(\Spec S)$ given by
restriction of scalars (so $[P]_R\mapsto [P]_S$).  
\item Fix a commutative, Noetherian ring $R$, and let $\fEt/R$ be the
subcategory of $\Aff$ consisting of objects $\Spec S$ where $R\ra S$
is finite \'etale.  Then $\Spec S\mapsto \GW(S)$ has the
structure of a Gysin functor, as detailed in Section~\ref{se:GWcat}.  
\item Let $\cC$ be the category of topological spaces, with morphisms
the quasi-finite fibrations.  Define $E(X)=\Hom(\pi_0(X),\Z)=H^0(X)$.
The pullback maps are as expected.  For $f\colon X\ra Y$ and
$\alpha\in E(X)$ define
$f_!(\alpha)$ to be the assignment $[y]\mapsto \sum_{x\in f^{-1}(y)}
\alpha([x])$ where $[x]$ and $[y]$ denote the path-components
containing $x$ and $y$.  This is a Gysin functor (the fibration condition is
needed only to show that $f_!$ is well-defined).   Note that this
Gysin functor has a strong relation to that in (a) above.   
\item The following is not an example of a Gysin functor, but is
nevertheless instructive.  Let $\cC$ be the category of finite sets,
and let $\cP(X)$ be the powerset of the set $X$; this is not quite a
ring, but it does have the intersection operation $\cap$ which we will
regard as a multiplication.    Given $f\colon X\ra
Y$ one has the inverse-image map $f^*\colon \cP(Y)\ra \cP(X)$ (which
preserves the multiplication) and the image map $f_*\colon \cP(X)\ra
\cP(Y)$ (which does not).  The axioms of Definition~\ref{de:Gysin} are all
satisfied, when suitably interpreted.  The powerset functor is
something like a ``non-additive Gysin functor''.  
\end{enumerate}
\end{example}

\begin{remark}
Let $\cC$ be the category of oriented topological manifolds, and let
$E(X)=H^*(X)$.  With the usual pullbacks and Gysin morphisms, this is
{\it almost\/} (but not quite) a Gysin functor as we defined above.
The difficulty is that the push-pull axiom only holds for pullback
squares satisfying a suitable transversality condition.  This same
problem arises if one uses smooth algebraic varieties and the Chow
ring, or if one uses oriented manifolds and complex cobordism.
But all of these settings represent appearances in the literature of
structure similar to what we consider in the present paper.
Especially in the case of cobordism, see the axiomatic treatment 
in \cite[Section 1]{Q}.  Prior to \cite[Proposition 1.12]{Q} Quillen
refers to, but does not give, an axiomatic treatment related to the
multiplicative structure; the axioms for a Gysin functor are essentially
 this.  
\end{remark}

\medskip

\subsection{The universal Gysin functor}

Given an object $X$ in $\cC$, define $\cA_\cC(X)$ to be the Grothendieck
group of (isomorphism classes of) maps $S\ra X$ where 
\[ [(S\amalg T)\ra X] = [S\ra X] + [T\ra X].
\]
The multiplication $[S\ra X]\cdot [T\ra X]=[S\times_X T\ra X]$ is
well-defined and makes $\cA_\cC(X)$ into a commutative ring with identity
$[\id\colon X\ra X]$.  We call $\cA_\cC(X)$ the \dfn{Burnside ring} of
$X$.  Note that $\cA_\cC$ has the evident structure of a contravariant
functor to rings, as well as that of a covariant structure to abelian
groups, generalizing the situation in Example~\ref{ex:Gysin}(b).
One readily checks that this is a Gysin functor, called the
\dfn{Burnside functor} for the category $\cC$.  When the category
$\cC$ is understood we abbreviate $\cA_\cC$ to just $\cA$.  

\begin{example}
When $\cC$ is the category of finite sets, note that there is a
natural isomorphism $\cA(S)\iso \Hom(S,\Z)$, sending the element
$[f\colon M\ra S]$ to the assignment $s\mapsto \#f^{-1}(s)$.  The Gysin functor
given in Example~\ref{ex:Gysin}(a) (restricted to the category of finite
sets) is the Burnside functor for this category.  
\end{example}

The Burnside functor has the following universal property:

\begin{prop}
\label{pr:Gysin-universal}
If $E$ is a Gysin functor on the category $\cC$ then there is a unique
map of
Gysin functors $\cA_\cC \ra E$.  It sends $[f\colon A\ra X]$ in
$\cA_\cC(X)$ to $f_!(1)\in E(X)$.  
\end{prop}

\begin{proof}
An easy exercise.  For existence, use the given formula.  The fact that $\cA_\cC(X)\ra E(X)$ is well-defined
follows using Lemma~\ref{le:Gysin-sum} (which implies that $(f\amalg
g)_!(1)=f_!(1)+g_!(1)$).  The fact that it is a ring map follows
from Lemma~\ref{le:Gysin-product}.  Compatibility with pullbacks and
pushforwards is trivial.  Uniqueness follows from the fact that
$[f\colon A\ra X]$ equals $f^\cA_!(1)$, the pushforward in the Gysin 
functor $\cA$.  
\end{proof}

\medskip

\subsection{Categories derived from Gysin functors}
Given a Gysin functor $E$ on $\cC$ 
we can define an additive category $\cC_E$ as follows.  First,
the objects of $\cC_E$ are the same as the objects of $\cC$.  Second,
for any objects $A$ and $B$ define 
\[ \cC_E(A,B)=E(B\times A).
\]
Really what we mean here is that $\cC_E(A,B)$ is the underlying
abelian group of $E(B\times A)$.  
Third, define the composition law
\[ \mu_{C,B,A}\colon \cC_E(B,C)\tens \cC_E(A,B) \ra \cC_E(A,C)
\]
by
\[ \mu_{C,B,A}(\alpha\tens \beta)=(\pi_{13})_!\bigl ( (\pi_{12})^*(\alpha) \cdot
{\pi_{23}}^*(\beta) \bigr )
\]
where the $\pi_{rs}$ maps are  the evident ones
\[ \pi_{12}\colon A\times B\times C \ra A\times B, \quad
\pi_{23}\colon A\times B\times C \ra B\times C,
\]
\[
\pi_{13}\colon A\times B\times C\ra A\times C.
\]
We will use the notation
\[ \alpha\circ \beta=\mu_{C,B,A}(\alpha\tens \beta).
\]
Finally, for any object $A$ define $i_A$ to be ${\Delta^A}_!(1)$; that
is, consider the map
\[ E(A) \llra{{\Delta^A}_!} E(A\times A)
\]
and take the image of the unit element of the ring $E(A)$.  Note that
$E(A\times A)$ is a commutative ring and so has a unit element $1$,
but this is not necessarily equal to $i_A$.  One may check (see
Proposition~\ref{pr:Gysin->Cat} below) that this structure makes
$\cC_E$ into a category.

For lack of a better term, we refer to elements of $E(B\times A)$ as
``$E$-correspondences'' from $A$ to $B$.  The category $\cC_E$ itself
will be referred to as  the \mdfn{category of $E$-correspondences}.

\begin{remark}
The construction of the category $\cC_E$ is one that appears countless
times in the algebraic geometry literature, ultimately going back to
Grothendieck.  For the category of algebraic varieties over some field
$k$, forming the category of correspondences with respect to the Chow
ring functor is the first step in Grothendieck's attempts to define a
category of motives.  See for example \cite[Section 2]{M}.
\end{remark}

\begin{example}\mbox{}\par
\begin{enumerate}[(a)]
\item Let $G$ be a finite group, let $\cC$ be the category of
finite $G$-sets, and let $\cA$ be the Burnside functor from
Example~\ref{ex:Gysin}(b).
The category $\cC_\cA$ is precisely the category $\Mackey$
mentioned in Section~\ref{se:intro}.
\item Fix a commutative, Noetherian ring $R$, and let $\cC$ be the
subcategory of $\Aff$ consisting of objects $\Spec S$ where $R\ra S$
is sheer and separable.  Then $\cC_{GW}$ is the Grothendieck-Witt
category over $R$, defined in Section~\ref{se:GWcat}.
\item Let $\cC$ be the category of finite sets, and let $E$ be the
Gysin functor from Example~\ref{ex:Gysin}(a).  Then we obtain the
category of correspondences $\cC_E$.   It turns out this category has
a familiar model: it is equivalent to the category of
finitely-generated, free abelian groups.  Proving this is not hard,
but it will also fall out of our general ``reconstruction theorem''
(Theorem~\ref{th:recon}).   See Example~\ref{ex:recon-ab}.
\end{enumerate}
\end{example}

The following proposition details many (and perhaps too many) useful
facts about the category $\cC_E$.  Recall one piece of notation: maps
into products can be unlabelled if there is a self-evident candidate
for how the map projects onto each of the factors.  For example,
if
$f\colon A\ra B$ then $A\ra A\times B$ denotes the evident map
that is the identity on the first factor and $f$ on the second.  

\begin{prop} 
\label{pr:Gysin->Cat}
Suppose given a Gysin functor $E$ on the category $\cC$.
\begin{enumerate}[(a)]
\item The structure described above defines a category $\cC_E$ that is
enriched over abelian groups, where $i_A\in \cC_E(A,A)$ is the
identity map on $A$.  
\item A natural transformation of Gysin functors $E\ra E'$ induces a
functor $\cC_E \ra \cC_{E'}$.  
\item There is a functor $R\colon \cC \ra \cC_E$ that is the identity
on objects and has the property that for $f\colon A\ra B$ in $\cC$ we
have
\[ R_f=(\id_B\times f)^*(i_B)\in E(B\times A)=\cC_E(A,B).
\]
One also has $R_f=(A\ra B\times A)_!(1)$.
\item 
The
category $\cC_E$ has an anti-automorphism $(\blank)^*$ that is the identity on
objects, and for $\alpha \in \cC_E(A,B)$ is given by
\[ \alpha^*=t^*(\alpha) \]
where $t\colon A\times B\ra B\times A$ is the evident isomorphism.
We define $I\colon \cC^{op}\ra \cC_E$ to be the identity on objects,
and to be given on maps by $I(f)=(R_f)^*$.  We often write
$I_f=I(f)$.  If $f\colon A\ra
B$ then $I_f=(f\times \id_B)^*(i_B)=(A\ra A\times B)_!(1)$.  
\item Suppose given $\alpha\in \cC_E(W,Z)$, $f\colon Y\ra W$, $g\colon
Z\ra U$, $f'\colon W\ra Y$,and $g'\colon U\ra Z$.  Then
\begin{enumerate}[(i)]
\item $\alpha\circ R_f=(\id_Z\times f)^*(\alpha)$;
\item $R_g\circ \alpha=(g\times \id_W)_!(\alpha)$;
\item $\alpha\circ I_{f'}=(\id_Z\times f')_!(\alpha)$;
\item $I_{g'}\circ \alpha=(g'\times \id_W)^*(\alpha)$.
\end{enumerate}

\item Given $A\llra{f} B \llla{q} C$ in $\cC$ one has $I_f\circ R_q=(f\times
q)^*(i_B)=(\pi^{A\times_B C}_{AC})_!(1)$ in $\cC_E$.
\item Given $A\llla{p} D \llra{g} C$ in $\cC$ one has $R_p\circ
I_g=(p\times g)_!(i_D)=\bigl ((p\times g)^{D}_{AC}\bigr )_!(1)$.
in $\cC_E$.  

\item Given a pullback diagram in $\cC$
\[ \xymatrix{
Z\ar[r]^g \ar[d]_p & W \ar[d]^q \\
X\ar[r]^f & Y
}
\]
one has $R_p\circ I_g=I_f\circ R_q$ in
$\cC_E$.  

\item If $f$ is an isomorphism in $\cC$ then $R_f=I_f^{-1}=I_{f^{-1}}$.  
\end{enumerate}
\end{prop}

\begin{proof}
This proof is tedious, but completely formal.  See Appendix~\ref{se:leftover}.
\end{proof}

Our next goal is to observe that the the Gysin functor $E$, which is
both co- and contravariant, extends to a single functor defined on all of
$\cC_E$.  Before embarking on the explanation of this, here is some
useful notation.
If $B$ is an object of $\cC$, note that the abelian group $E(B)$ may be
identified with both $\cC_E(*,B)$ and $\cC_E(B,*)$.  If $x\in E(B)$ we
write $x_*$ for $x$ regarded as an element of $\cC_E(*,B)=E(B\times
*)$ and $\pres{x}$ for $x$ regarded as an element of $\cC_E(B,*)$.  
This notation makes sense if one remembers our general
``right-to-left'' notation; e.g., $x_*$ is $x$ regarded as a map {\it
from} the object $*$.  

Define a functor $E'\colon \cC_E^{op}\ra \Ab$ as follows.  On
objects it is the same as $E$: $E'(A)=E(A)$.  For $g\in \cC_E(A,B)$
define $E'(g)\colon E(B)\ra E(A)$ by
\[ E'(g)(x)= \pres{x} \circ g = \bigl(\pi^{BA}_A\bigr)_!\Bigl[ \bigl (\pi^{BA}_B\bigr)^*(x)
\cdot g \Bigr ].
\]
The fact that this is a functor is immediate from the associativity
and unital properties of the circle product $\circ$ 
(the composition product in $\cC_E$).  

\begin{prop}
\label{pr:Gysin-extend}
The functor $E'\colon \cC_E^{op}\ra
\Ab$ has the property that
$E'(Rf)=f^*$ and $E'(If)=f_!$ for any map $f$ in $\cC$.   
\end{prop}

\begin{proof}
Immediate from Proposition~\ref{pr:Gysin->Cat}(e), parts (i) and
(iii).
\end{proof}

\begin{remark}
Note that $E'$ is not a functor from $\cC_E^{op}$ into $\CRing$.
This would of course be too much to ask, since the transfer maps
$f_!$ do not respect the multiplicative products.
\end{remark}

\subsection{Further properties of \mdfn{$\cC_E$}}
The category $\cC_E$ has some extra structure that we have not yet
accounted for.  The categorical product in $\cC$ induces a symmetric
monoidal product on $\cC_E$: that is, for objects $X$ and $Y$ we
define
\[ X\tens Y = X\times_{\cC} Y.
\]
We must define $f\tens g$ for $f\in \cC_E(X,X')$ and $g\in
\cC_E(Y,Y')$.  We do this by
\[ f\tens g= \bigl( t^{X'Y'XY}_{X'XY'Y}\bigr)^*(f\tens g).
\]
This formula appears self-referential, but the two tensor symbols mean
something different: in the second case, we have $f\in E(X'\times X)$
and $g\in E(Y'\times Y)$ and $f\tens g$ is the element in $E(X'\times
X\times Y'\times Y)$ that was introduced in Defintion~\ref{de:Gysin}.   
It takes a little work to verify bi-functoriality.  
The unit object is $S=*$, the terminal object of $\cC$ (note that this
is not a terminal object of $\cC_E$).  The symmetry isomorphism
$\tau_{XY}\in \cC_E(X\tens Y,Y\tens X)$ is defined to be
\[ \tau_{XY}=R(t_{XY}) \]
where $t_{XY}\colon X\times Y\ra Y\times X$ is the canonical
isomorphism in $\cC$.  One must verify that the structure we have
defined  satisfies the basic
commutative diagrams for a symmetric monoidal structure, and we again
leave this with simply the remark that it is tedious but not challenging.

We can also define function objects in $\cC_E$.  For objects $X$ and
$Y$ define
\[ F(X,Y)=X^*\tens Y
\]
where $(\blank)^*$ is the anti-automorphism from
Proposition~\ref{pr:Gysin->Cat}(d). Of course the object $X^*$ is
exactly equal to $X$, but we wrote $X^*$ because this is more
compatible with the way the maps work:
for $g\colon Y\ra Y'$ define
$F(X,g)$ to be the map $i_{X^*}\tens g$, and for $f\colon X\ra X'$ define
$F(f,Y)$ to be the map $I_f\tens i_Y$.  

At this point it is useful to recall the notion of dualizability in
symmetric monoidal categories.  See Appendix~\ref{se:dual}.  In this
paper we will use the term {\bf tensor category\/} to signify a
symmetric monoidal category that is also enriched over abelian groups,
having the property that the tensor product of morphisms is bilinear.
A tensor category is {\it closed\/} if it is equipped with function
objects related to the tensor by the usual adjunction formula, which
is required to be linear.

With the above notions in place, we leave the reader to check the following:

\begin{prop}
The above structure makes $\cC_E$ into a closed tensor category
in which every object is dualizable.
Moreover, every object is isomorphic to its own dual.  
\end{prop}

\begin{proof}
Tedious, but routine.  Perhaps the only thing that needs remark is
that the evaluation and co-evaluation morphisms for an object $X$ are
\[ \cev_X=i_X\in E(X\times X)=E(X\times X\times *)=\cC_E(*,X\times
X)=\cC_E(S,X\tens X)
\]
and
\[ \ev_X=i_X\in E(X\times X)=E(*\times X\times X)=\cC_E(X\times
X,*)=\cC_E(X\tens X,S).
\]
\end{proof}

The following proposition is easy but important.  It will be used
implicitly in several later calculations.  

\begin{prop}
\label{pr:RI-tensor}
Let $f\colon A\ra B$ and $g\colon X\ra Y$ be maps in $\cC$.  Then
$R(f\times g)=Rf\tens Rg$ and $I(f\times g)=If\tens Ig$.  
\end{prop}

\begin{proof}
Using Proposition~\ref{pr:Gysin->Cat}(d) and the definition of tensor product,
we have
\begin{align*}
 If\tens Ig & =(t^{AXBY}_{ABXY})^* 
\Bigl [ 
(A \ra AB)_!(1)\tens (X\ra XY)_!(1)
\Bigr ]\\
& =(t^{AXBY}_{ABXY})^* \Bigl [
(A\times X\ra A\times B\times
X\times Y)_!(1)\Bigr ]
\\
&= (A\times X\ra A\times X\times B\times Y)_!(1) \\
&= I(f\times g).
\end{align*}
The second equality uses the Push-Product Axiom, the third equality
uses Push-Pull, and the last equality is
Proposition~\ref{pr:Gysin->Cat}(d) again.
\end{proof}

We close this section by returning to the functor $E'$ from
Proposition~\ref{pr:Gysin-extend}.  The following proposition is not needed, but we record it for
completeness.  The proof is left to the reader.

\begin{prop}
Let $\cC$ be a finitary lextensive category, and let $\cA$ be the
Burnside functor for $\cC$.  Let $E$ be a Gysin functor on $\cC$.  
\begin{enumerate}[(a)]
\item 
The unit maps $\Z\ra E(X)$ and pairings $E(X)\tens E(Y)\ra E(X\tens
Y)$ provide $E'$ with the structure of lax symmetric monoidal functor.

\item The association $E\mapsto E'$ gives
a bijection between Gysin functors on $\cC$ and lax symmetric monoidal
functors $\cC_{\cA}^{op}\ra \Ab$.  
\end{enumerate}
\end{prop}



\section{Gysin schema and the reconstruction theorem}
\label{se:recon}

We have seen that given a Gysin functor $E$ on a finitary lextensive
category $\cC$, there is an associated symmetric monoidal category
$\cC_E$ called the category of $E$-correspondences.  One could try to
run this process in reverse: given a symmetric monoidal category
$\cD$, what do you need to know in order to guarantee that $\cD$ is
the category of $E$-correspondences for an appropriately chosen $E$
and $\cC$?
We might term this the ``reconstruction problem'': can $\cD$ be
reconstructed as a category of correspondences?  Of course for this to
work one must at
least require that all objects in $\cD$ be self-dual.

Unfortunately, in this form the reconstruction problem is a little
awkward.  The category $\cC_E$ comes equipped with two distinguished
subcategories, one consisting of the forward maps $Rf$ and one
consisting of the backward maps $If$.  If we are just given a
symmetric monoidal category $\cD$, there is no clear way to separate
out analogs of either of these distinguished subcategories.

The way around this problem is to add these special subcategories into
the initial data.  Then the reconstruction problem becomes solvable,
albeit for almost
tautological reasons.  See Theorem~\ref{th:recon} below.

\subsection{Gysin schema}

\begin{defn}
\label{de:schema}
A \dfn{Gysin schema} consists of the following data:
\begin{enumerate}[$\bullet$]
\item A finitary lextensive category $\cC$, together with an explicit
choice $*$ for terminal object and for each objects $X$ and $Y$ of
$\cC$ an explicit choice of product $X\times Y$;
\item A tensor category $(\cD,\tens,S)$;
\item A map of sets $\Theta\colon \ob \cC\ra \ob \cD$ and two functors
$R\colon \cC\ra \cD$ and $I\colon \cC^{op}\ra \cD$;
\item Isomorphisms $\theta_*\colon \Theta(*)\ra S$ and $\theta_{X,Y}\colon \Theta(X\times Y)\ra (\Theta
X)\tens (\Theta Y)$.
\end{enumerate}
This data is required to satisfy the following axioms:
\begin{enumerate}[(1)]
\item $R(X)=\Theta(X)=I(X)$ for all objects $X$ of $\cC$;
\item The data $(R,\theta)$ makes $R$ into a strong symmetric monoidal
functor from $(\cC,\times,*)$ to $(\cD,\tens,S)$.
\item For all maps $f\colon A\ra X$ and $g\colon B\ra Y$ in $\cC$, the
diagram
\[ \xymatrixcolsep{3.5pc}\xymatrix{
I(A\times B)\ar[d]_{\theta_{A,B}}^\iso 
& I(X\times Y) \ar[l]_{I(f\times g)}\ar[d]^{\theta_{X,Y}}_\iso \\
I(A)\tens I(B) & I(X)\tens I(Y) \ar[l]_{I(f)\tens I(g)}
}
\]
is commutative.
\item For every pullback diagram
\[ \xymatrix{
A \ar[d]_p \ar[r]^f & B\ar[d]^q \\
C \ar[r]^g & D
}
\]
in $\cC$ one has $Rf\circ Ip=Iq\circ Rg$.  
\end{enumerate}
We will write the Gysin schema as $\Theta\colon \cC\ra \cD$,
suppressing $R$, $I$, and $\theta$ from the notation.  
\end{defn}

\begin{remark}
There are a couple of odd features about the above definition.  First,
the function $\Theta$ is clearly redundant as it can be recovered from
either $R$ or $I$.  We include $\Theta$ in the definition because it
is often useful to have a notation that does not favor either $R$
or $I$.  Secondly, conditions (2) and (3) could have been made more symmetric
by replacing (3) with the statement that $(I,\theta)$ is strong
symmetric monoidal; we leave the equivalence as an exericse.  The
phrasing from the definition makes applications a little easier, as
there is a bit less to verify: in practice one looks for a ``nice enough''
functor $R$ that admits transfer maps satisfying (3) and (4).   
\end{remark}

\begin{example} 
\label{ex:Gysin-schema}
One readily checks that the following are examples of
Gysin schema:
\begin{enumerate}[(a)]
\item Fix a finite group $G$, and let $\cD$ be the $G$-equivariant
stable homotopy category of genuine $G$-spectra.  Let $\cC$ be the
category of finite $G$-sets, and let $R(X)=\Sigma^\infty(X_+)$.  The
maps $I(f)$ are the usual transfer maps constructed in stable homotopy
theory.  
\item Let $\cD$ be the category of finitely-generated free abelian
groups, equipped with the tensor product.  Let $\cC$ be the category
of finite sets.  Let $R(X)$ be the free abelian group on the set $X$,
with its natural functoriality.  If $f\colon X\ra Y$ then let
$I(f)\colon R(Y)\ra R(X)$ send the basis element $[y]$ to $\sum_{x\in
  f^{-1}(y)} [x]$.  
\end{enumerate}
\end{example}

When $X$ is an object of $\cC$ we let $\pi_X$ denote the unique map
$X\ra *$, and $\Delta_X$ denote the diagonal $X\ra X\times X$.  The
subscripts will usually be suppressed when understood.  Note that
$R\pi$ is a map $RX\ra R(*)$, and we have a chosen isomorphism
$R(*)=\Theta(*)\iso S$; so composing these gives a canonical map $RX\ra
S$, which we will usually also denote $R\pi$ by abuse.  Similarly,
$R\Delta$ may be regarded as a map $RX\ra RX\tens RX$.  We use these
conventions for $I\pi$ and $I\Delta$ as well.

\subsection{Transfers and duality}

\begin{prop}
\label{pr:I-dual}
Suppose that $\Theta\colon \cC\ra \cD$ is a Gysin schema.  Then for
every object $X$ in $\cC$, $\Theta X$ is dualizable in $\cD$.  In
fact, $\Theta X$ is self-dual with 
structure maps given by
\[ \xymatrix{
S \ar[r]^-{I\pi} & \Theta(X) \ar[r]^-{R\Delta} & \Theta(X\times X)
\ar[r]^\iso & \Theta X\tens \Theta X
}
\]
and
\[ \xymatrix{
\Theta X\tens \Theta X \ar[r]^\iso & \Theta(X\times X) \ar[r]^-{I\Delta}
& \Theta X \ar[r]^{R\pi} & S.
}
\]
\end{prop}

\begin{proof}
The key is the pullback diagram 
\[ \xymatrixcolsep{3pc}\xymatrix{
X\ar[r]^\Delta \ar[d]_\Delta & X\times X \ar[d]^{\Delta\times \id} \\
X\times X\ar[r]^-{\id\times \Delta} & X\times X\times X,
}
\]
from which we deduce that $I(\id\tens \Delta)\circ R(\Delta\tens
\id)=R\Delta \circ I\Delta$.  Combining this with axiom (3) from
Definition~\ref{de:schema} gives the first equality below:
\begin{myequation}
\label{eq:I-R}
 (\id\tens I\Delta)\circ (R\Delta\tens \id)=R\Delta\circ I\Delta =
(I\Delta\tens \id)\circ (\id\tens R\Delta).
\end{myequation}
The second equality comes about in the same way, but starting with the
reflection of the above pullback square about its central diagonal.

To prove the proposition we must first check that the composition
\[ \xymatrix{
\Theta X=\Theta X\tens S \ar[r]^-{1\tens I\pi} & \Theta X\tens \Theta X
\ar[r]^-{1\tens R\Delta} &
\Theta X\tens \Theta X\tens \Theta X \ar[r]^-{I\Delta \tens 1}&
\Theta X\tens \Theta X \ar[d]_{R\pi\tens 1} 
\\ &&&S\tens \Theta X =\Theta X
}
\]
equals the identity.  But using (\ref{eq:I-R}) this is equal to
\[ (R\pi\tens 1)\circ R\Delta \circ I\Delta \circ I(1\tens \pi)= R((\pi
\times 1)\circ \Delta) \circ I((1\times \pi)\circ \Delta) =
R(\id)\circ I(\id)=\id.
\]
Note that we have again used axiom (3) of Definition~\ref{de:schema}.

The proof that the composite
\[ \xymatrix{
\Theta X=S\tens \Theta X \ar[r]^-{I\pi\tens 1} & \Theta X\tens \Theta X
\ar[r]^-{R\Delta\tens 1} &
\Theta X\tens \Theta X\tens \Theta X \ar[r]^-{1\tens I\Delta}&
\Theta X\tens \Theta X \ar[d]_{1\tens R\pi} 
\\ &&&S\tens \Theta X =\Theta X
}
\]
equals the identity is entirely similar.  
\end{proof}

The following corollary is also worth recording:

\begin{cor}
Let $\Theta\colon \cC\ra \cD$ by a Gysin schema.  Then given any map
$f\colon X\ra Y$ in $\cC$, the dual of $Rf\colon \Theta X\ra \Theta Y$
(computed using the duality structures provided by Proposition~\ref{pr:I-dual})
is precisely $If\colon \Theta Y\ra \Theta X$.
\end{cor}

\begin{proof}
The dual of $Rf$ is the following composite:
\[\xymatrixcolsep{3pc}\xymatrix{
\Theta Y=\Theta Y\tens S \ar[r]^{1\tens \eta_X} & \Theta Y\tens \Theta X\tens \Theta X
\ar[r]^{1\tens Rf\tens 1} & \Theta Y\tens \Theta Y\tens \Theta X
\ar[d]^{\epsilon_Y\tens 1} \\
&& S\tens \Theta X=\Theta X.
}
\]
One unpacks $\eta$ and $\epsilon$ as $\eta_X=R\Delta_X\circ I\pi_X$
and $\epsilon_Y=R\pi_Y\circ I\Delta_Y$, and then argues precisely as
in the proof of Proposition~\ref{pr:I-dual} but instead using the
pullback diagram
\[ \xymatrixcolsep{3pc}\xymatrix{
X\ar[r]^{f\times 1} \ar[d]_{f\times 1} & Y\times X \ar[d]^{1\times f\times 1} \\
Y\times X\ar[r]^-{\Delta_Y\times 1} & Y\times Y\times X,
}
\]
The details are left to the reader.
\end{proof}

\subsection{The canonical Gysin functor for a Gysin schema}

Suppose $\Theta\colon \cC\ra \cD$ is a Gysin schema.
Define $\piS\colon \cC^{op}
\ra \Ab$ to be the functor given by
\[ \piS(X)=\cD(\Theta X,S).
\]
Note that the abelian groups $\piS(\blank)$
also inherit the structure of a covariant functor:  given $f\colon
X\ra Y$ in $\cC$ define $f_!\colon \piS(X)\ra \piS(Y)$ by the diagram
\[
\xymatrixcolsep{3pc}\xymatrix{
 \piS(X)\ar[r]^{f_!}\ar@{=}[d]&
\piS(Y)\ar@{=}[d]
\\
\cD(\Theta X,S) \ar[r]^{\cD(If,S)} 
& \cD(\Theta Y,S).
}
\]
Moreover, the abelian groups $\piS(X)$ inherit a product: given $a,b\in
\piS(X)$, define $a\cdot b$ to be the composite
\[ \Theta X\llra{R\Delta} \Theta X\tens \Theta X \llra{a\tens b} S\tens S\iso S.
\]
This gives $\piS(X)$ the structure of a commutative ring, and if $f\colon
X\ra Y$ is
a map in $\cC$ then $f^*\colon \piS(Y)\ra \piS(X)$ is a ring
homomorhism.

\begin{prop} If $\Theta\colon\cC\ra\cD$ is a Gysin schema then $\piS\colon
\cC^{op}\ra \CRing$ is a Gysin functor.
\end{prop}

\begin{proof}
This is simply a matter of chasing through definitions.
\end{proof}

\begin{remark}
Recall from Definition~\ref{de:Gysin} that if $a\in \piS(X)$ and $b\in \piS(Y)$ then we have the
element $a\tens b\in \piS(X\times Y)$.  It is easy to check that this
is the map
\[\xymatrix{
 \Theta(X\times Y)\ar[r]^{\iso} & \Theta X\tens \Theta Y  \ar[r]^{a\tens b}
& S\tens S =S.
}
\]
\end{remark}

\subsection{Preliminaries on the reconstruction problem}
\label{se:cat->Gysin}

Let $(\cD,\tens,S,F(\blank,\blank))$ be a closed, symmetric monoidal
category in which every object is dualizable.  It turns out all such
categories have a description that is somewhat reminiscent of the
construction of $\cC_E$.

For an object $X$ write $X^*=F(X,S)$, and for $f\colon X\ra Y$ write
$f^*=F(f,S)$.  Let $\ev_X\colon X^*\tens X\ra S$ be the adjoint of the
identity map $X^*\ra F(X,S)$, and let $c_X\colon S\ra X\tens X^*$ be
the coevaluation map guaranteed by duality (see Appendix~\ref{se:dual}).

Define a new category $\cD^{ad}$ as follows.  The objects are the same
as those in $\cD$, and morphisms are given by
\[ \cD^{ad}(X,Y)=\cD(Y^*\tens X,S).
\]
If $\alpha\in \cD^{ad}(X,Y)$ and $\beta\in \cD^{ad}(Y,Z)$ then
$\beta\circ\alpha$ is given as follows:
\[\xymatrixcolsep{2.6pc}\xymatrix{
Z^*\tens X \ar[r]^-\iso & Z^*\tens S\tens X \ar[r]^-{1\tens c_Y\tens 1} & Z^*\tens Y\tens
Y^*\tens X \ar[r]^-{\beta\tens \alpha} & S\tens S \ar@{=}[r] & S.
}
\]
One readily checks that this composition is associative, and $\ev_X\in
\cD^{ad}(X,X)$ is a two-sided identity.  

There is a functor $\Gamma\colon \cD\ra \cD^{ad}$ defined as follows.
It is the identity on objects, and given $f\colon X\ra Y$ we let
$\Gamma f\in \cD^{ad}(X,Y)=\cD(Y^*\tens X,S)$ be the composite
\[ \xymatrix{
 Y^*\tens X \ar[r]^{\id\tens f} & Y^*\tens Y \ar[r]^-{\ev_Y} & S.
}
\]
The check that this is indeed a functor is best done using the
graphical calculus for closed symmetric monoidal categories (see
\cite{BS} for an expository account of this).  If
$f\colon X\ra Y$ and $g\colon Y\ra Z$ are maps in $\cD$, then
$\Gamma(gf)$ and $(\Gamma g)(\Gamma f)$ are the composite maps
represented by the following diagrams:

\begin{picture}(300,130)(0,-30)
\put(30,0){\includegraphics[scale=0.4]{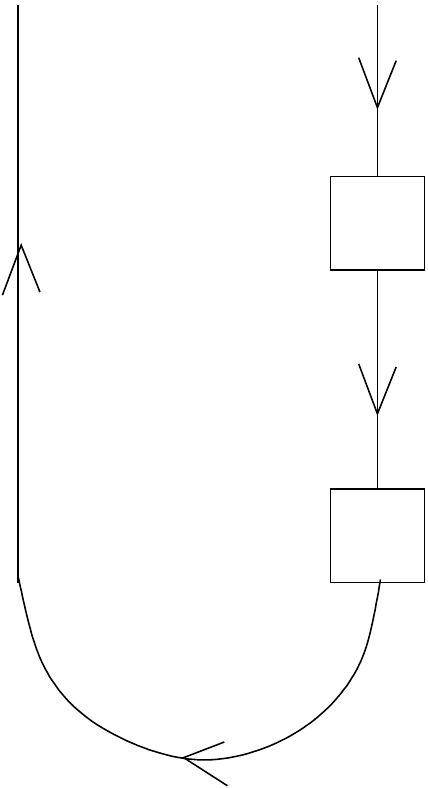}}
\put(180,0){\includegraphics[scale=0.4]{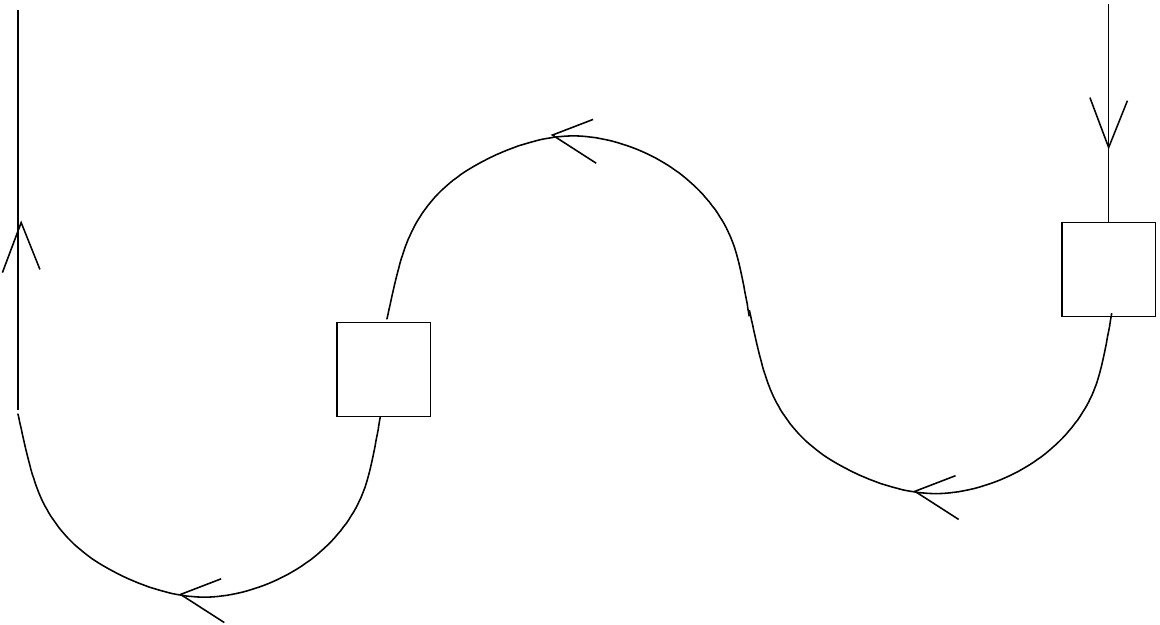}}
\put(72,64){$\scriptstyle{f}$}
\put(72,28){$\scriptstyle{g}$}
\put(306,39.3){$\scriptstyle{f}$}
\put(222,28){$\scriptstyle{g}$}
\put(35,-13){$\Gamma(g\circ_\cD f)$}
\put(220,-13){$\Gamma(g)\circ_{\cD^{ad}}\Gamma(f)$}
\end{picture}

\noindent
The graphical calculus clearly shows these composites to be identical in
$\cD$.

\begin{prop}
\label{pr:D-ad}
The functor $\Gamma\colon \cD\ra \cD^{ad}$ is an isomorphism of
categories.  
\end{prop}

\begin{proof}
We only need check that the maps $\cD(X,Y)\ra
\cD^{ad}(X,Y)=\cD(Y^*\tens X,S)$ are bijections.  There is an evident
map in the opposite direction that sends a map $h\colon Y^*\tens X\ra S$ to
the composite
\[
\xymatrixcolsep{3pc}\xymatrix{
X \ar@{=}[r] & S\tens X \ar[r]^-{c_Y\tens \id_X} & 
Y\tens Y^*\tens X \ar[r]^-{\id_Y\tens h} & Y\tens S \ar@{=}[r] & Y.
}
\]
Proving that these assignments are inverses to each other 
is another exercise in graphical
calculus.  For example, the composite in one direction sends the map $f\colon X\ra Y$
to the map represented by 

\begin{picture}(300,80)(0,-10)
\put(100,0){\includegraphics[scale=0.4]{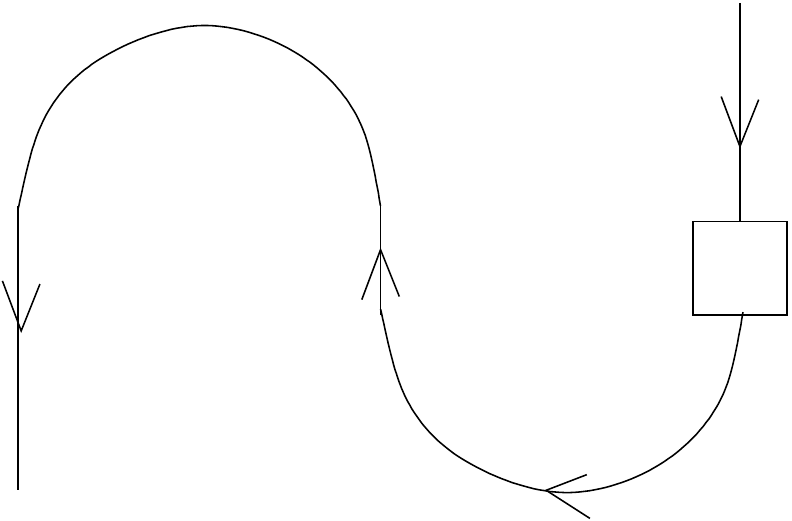}}
\put(184,28){$\scriptstyle{f}$}
\end{picture}

\noindent
and the graphical calculus shows that this is equal to $f$ in
$\cD$.  The other direction is similarly easy.
\end{proof}

Now suppose that $(\cD,\tens,S)$ is a symmetric monoidal category (not
necessarily closed) but where all objects are self-dual: assume that
for every object $X$ in $\cD$ one is supplied maps $\eta_X\colon S\ra
X\tens X$ and $\epsilon_X\colon X\tens X\ra S$ satisfying the
conditions of Definition~\ref{de:dual}. 
 In this context one can reproduce the
construction of $\cD^{ad}$ but without any explicit mention of duals.
Specifically, define $\cD^{\fad}$ to be the category with the same
objects as $\cD$, but where $\cD^{\fad}(X,Y)=\cD(Y\tens X,S)$.  
Given $f\in \cD^{\fad}(X,Y)$ and $g\in \cD^{\fad}(Y,Z)$ define $g\circ
f$ to be the composition
\[\xymatrixcolsep{3pc}\xymatrix{
Z\tens X \ar@{=}[r] & Z\tens S\tens X \ar[r]^-{1\tens \eta_Y\tens 1} &
 Z\tens Y\tens Y\tens X \ar[r]^-{g\tens f} & S\tens S=S.
}
\]
Let $1_X\in \cD^{\fad}(X,X)$ be the map $\epsilon_X\colon X\tens X\ra
S$.  One readily checks that $\cD^{\fad}$ is a category.

\begin{prop}
\label{pr:recon-1}
There is a functor $\Gamma\colon \cD\ra \cD^{\fad}$ that is the 
identity of objects and sends a map $f\colon X\ra Y$ in
$\cD$ to the composite
\[ \xymatrix{
Y\tens X \ar[r]^{\id\tens f} & Y\tens Y \ar[r]^-{\epsilon_Y} & S.
}
\]
The functor $\Gamma$ is an isomorphism of categories.
\end{prop}

\begin{proof}
A simple exercise.
\end{proof}

\subsection{The main reconstruction theorem}

Recall that $\cC_{(\piS)}$ denotes the category of correspondences
associated to the Gysin functor $\piS$.  

\begin{thm}
\label{th:recon}
Assume given a Gysin schema $\Theta\colon \cC\ra \cD$. 
Then there is full and faithful functor of 
categories $\cC_{(\piS)} \ra \cD$ that is the identity
on objects and sends a map $f\in \cC_{\piS}(X,Y)=\cD(Y\tens X,S)$ to
the composite
\[ \xymatrixcolsep{3pc}\xymatrix{
X \ar@{=}[r] & S\tens X \ar[r]^-{\eta_Y\tens \id} & Y\tens Y\tens X
\ar[r]^-{\id\tens f} & Y\tens S \ar@{=}[r] & Y.
}
\]
\end{thm}

\begin{proof}
The proof is easier to understand if we first compare $\cC_{(\piS)}$ to
$\cD^{\fad}$.  Note that the set of objects of these two categories
are identical, and for any objects $X$ and $Y$ we have equalities of
sets
\[ \cC_{(\piS)}(X,Y)= \piS(Y\tens X)=\cD(Y\tens X,S)=\cD^{\fad}(X,Y).
\]
The identity element $i_X\in \cC_{(\piS)}(X,X)=\piS(X\tens X)$ is
$\Delta_!(1)$, which unravelling the definitions equals the composite
\[ \xymatrix{
\Theta X\tens \Theta X\ar[r]^-{I\Delta} & \Theta X \ar[r]^{R\pi} & S,
}
\]
which equals $\epsilon_X$.  This is equal to the identity in
$\cD^{\fad}$.  

Finally, we must compare the composition rules in $\cC_{(\piS)}$ and
$\cD^{\fad}$.  Suppose given $f\in \cD(Y\tens X,S)$ and $g\in
\cD(Z\tens Y,S)$.  The composition $g\circ f$ in $\cC_{(\piS)}$ is given
by the composite
\[ \xymatrix{
\Theta Z\tens \Theta X \ar[r]_\theta^\iso & \Theta (Z\times X) 
\ar[r]^-{Ip} & \Theta (Z\times Y\times X) \ar[d]^-{R\Delta} \\
&& \Theta
(Z\times Y\times X)\tens \Theta(Z\times Y\times X) \ar[d]_\iso^\theta \\
&
& \Theta Z\tens \Theta Y\tens \Theta X \tens \Theta Z\tens \Theta
Y\tens \Theta X\ar[d]^{1\tens 1\tens R\pi_X\tens R\pi_Z\tens 1\tens 1}\\
&& \Theta Z\tens \Theta Y \tens S\tens S\tens \Theta Y\tens \Theta
X\ar[d]^{g\tens 1\tens 1\tens f} \\
&& S\tens S\tens S\tens S \ar@{=}[r] &S
}
\]
where $p\colon Z\times Y\times X\ra Z\times X$ is the evident
projection.
The composition $g\circ f$ in $\cD^{\fad}$ is given by the composite
\[ \xymatrix{
\Theta Z\tens \Theta X \ar[r]^-\iso &\Theta Z\tens S \tens \Theta X
\ar[r]^-{1\tens \eta_Y\tens 1} &
\Theta Z\tens \Theta Y\tens \Theta Y \tens \Theta X \ar[r]^-{g\tens f} & S\tens S
\ar@{=}[r]
& S.
}
\]
A diagram chase shows these two composites to be equal.  This is best
left to the reader, but main idea is to take the first composite and
decompose the diagonal on
$Z\times Y\times X$ into the three diagonals on the
individual components.  The diagonals on $Z$ and $X$ cancel the $\pi_X$
and $\pi_Z$ appearing later, leaving only the diagonal on $Y$.  The
map $Ip$ is equal to $1\tens I\pi_Y\tens 1$, and the $I\pi_Y$
assembles with the $R\Delta_Y$ to make $\eta_Y$.  

At this point we have constructed a functor $\cC_{(\piS)}\ra
\cD^{\fad}$.  It is readily seen to induce isomorphisms on the
Hom-sets, and so it is an isomorphism of categories. Finally, pair
this with Proposition~\ref{pr:recon-1} to get the desired result.
\end{proof}

\begin{example}
\label{ex:recon-ab}
Let $\cD$ be the category of finitely-generated free abelian groups,
and let $\cC$ be the category of finite sets.  Let $\Theta\colon \cC \ra
\cD$ be the free abelian group functor, given the structure of a Gysin
schema as in Example~\ref{ex:Gysin-schema}.  
The associated Gysin functor $\piS$ is precisely the one of
Example~\ref{ex:Gysin}(a).  By Theorem~\ref{th:recon} we conclude that
$\cC_{(\piS)}$ is isomorphic to the category of finitely-generated free
abelian groups.
\end{example}



\section{The structure of correspondence categories}
\label{se:struc}

Suppose $E\colon \cC\ra \CRing$ is a Gysin functor.  Our goal is to
better understand how the category of correspondences
$\cC_E$ relates to the original category $\cC$.  Given objects $X$ and
$Y$ in $\cC$, every element $f\in \cC(X,Y)$ gives rise to maps $R_f$
and $I_f$ in $\cC_E$.  In addition, we will see that every element
$a\in E(X)$ gives an endomorphism $D_a$ of $X$ in $\cC_E$.  We will
prove that every map in $\cC_E$ may be written in the form $R_f\circ
D_a\circ I_g$, and we will explain rules for rewriting the composition
of  two such expressions into the same form.

These results do not give a simple picture for the structure of
$\cC_E$, but they do give a reasonable prescription for working with
these categories in specific examples.  In Section~\ref{se:Gysin-Galois} we
explore this in a general ``Galoisien'' setting (where the category
$\cC$ has properties formally similar to the category of $G$-sets,
$G$ is a finite group).


\subsection{The diagonal structure}
We will need an extra piece of structure in $\cC_E$ coming  from the
diagonal maps in $\cC$.  
For an object $X$ in $\cC$ let $\Delta\colon X\ra
X\times X$ be the diagonal.  This induces a map of abelian groups
\[ D=\Delta_!\colon E(X) \ra E(X\times
X)=\cC_E(X,X).
\]
The target has two ring structures: it has the generic ring structure
that any $E(Z)$ has, and it has the circle product coming from
composition in $\cC_E$.  It is the latter that we wish to consider:

\begin{prop}
\label{pr:D}
$D\colon E(X)\ra \cC_E(X,X)$ is a ring map, and for any $a\in E(X)$
one has $(Da)^*=Da$.
\end{prop}  

\begin{proof}
Let $a,b\in E(X)$.  We calculate
\begin{align*}
Db\circ Da &= (\pi_{13})_!\bigl [ \pi_{12}^*(\Delta_!b) \cdot
\pi_{23}^*(\Delta_!a) \bigr ] \\
&= (\pi_{13})_!\bigl [ (\Delta\times \id)_!(\pi_1^*b) \cdot
\pi_{23}^*(\Delta_!a) \bigr ] \qquad\qquad\qquad(\text{push-pull})\\
&= (\pi_{13})_!(\Delta\times \id)_!\bigl [ (\pi_1^*b) \cdot
(\Delta\times\id)^*\pi_{23}^*(\Delta_!a) \bigr ] \qquad (\text{proj. formula})\\
&= \pi_1^*b \cdot \Delta_!(a) \\
&= \Delta_!\bigl ( (\Delta^* \pi_1^*b) \cdot a\bigr ) \\
&= \Delta_!(b\cdot a) \\
&= D(ba).
\end{align*}

The second statement in the proposition is proven by 
\begin{align*}
(Da)^*=t^*(Da)=(t_!)^{-1}(Da)=t_!(Da)=t_!(\Delta_! a)=\Delta_!a=Da.
\end{align*}
The second equality is from Proposition~\ref{pr:Gysin->Cat}(i), and
the third equality is because $t=t^{-1}$.    
\end{proof}

\begin{notation}
We will usually write $Da$, or if really necessary $D(a)$, but
sometimes we will write $D_a$ for the same thing.  
\end{notation}

\begin{prop}
\label{pr:R-D}
Suppose given $f\colon X\ra Y$ and $a\in E(Y)$.  Then $Da\circ
R_f=R_f\circ D(f^*a)$ and $I_f\circ Da=D(f^*a)\circ I_f$.  
\end{prop}

\begin{proof}
We compute
\begin{align*}
Da\circ R_f = \Delta_! a\circ R_f = (\id\times f)^*(\Delta_! a)
& = \bigl ((f\times \id)^{X}_{Y\times X}\bigr )_! (f^*a) \\
&=  (f\times \id)_!\Delta_! (f^*a) \\
&= R_f\circ D(f^*a).
\end{align*}
The second and fifth equalities are by
Proposition~\ref{pr:Gysin->Cat}(e), and the third equality is by
push-pull.

To conclude,
the second statement in the proposition follows by applying
$(\blank)^*$ to the first and using Proposition~\ref{pr:D}.
\end{proof}

\begin{remark}
\label{re:twisted-endo}
Let $G=\Aut_\cC(X)$, with the group structure coming from composition.
Note that there is a map $G^{op}\ra \Aut(E(X))$ given by $f\mapsto
f^*$.  
Let $E(X)[\widetilde{G}]$ be the twisted group ring defined as follows: it is spanned by
elements $a[f]$ for $a\in E(X)$ and $f\in G$, and the multiplication is
induced by
\[ a[f]\cdot b[g]=a\bigl ({(f^{-1}})^*b\bigr )[fg].
\]
Then Proposition~\ref{pr:R-D} shows that there is a map of rings
\[ E(X)[\widetilde{G}] \lra \cC_E(X,X), \qquad a[f]\mapsto Da\circ R_f.
\]
In good cases this is an isomorphism: see Proposition~\ref{pr:Galois-endo}
below.  
\end{remark}

\subsection{Initial results on the structure of $\cC_E$}

If we have maps $Y\llla{f} Z \llra{g} X$ and $a\in E(Z)$ then
$R_f\circ D_a\circ I_g$ is a morphism from $X$ to $Y$ in $\cC_E$.
We will refer to such an expression as an $RDI$ formula for the
composite morphism.  
Here are some useful facts that relate these $RDI$ formulas
in $\cC_E$ to pushforwards in $E$:

\begin{prop}
\label{pr:triple}
Suppose given maps $Y\llla{f} Z\llra{g} X$ and $a\in E(Z)$.
Then:
\begin{enumerate}[(a)]
\item
$R_f\circ D_a\circ I_g=\bigl((f\times
g)^{Z}_{YX}\bigr )_!(a)$.  
\item $R_f\circ D_a\circ I_f=D(f_!a)$.
\item $R_f\circ I_f=D(f_!1)$.  
\end{enumerate}
\end{prop}

\begin{proof}
For (a) we use Proposition~\ref{pr:Gysin->Cat}(e) to write
\[ R_f\circ D_a\circ I_g=(f\times \id)_! (\id\times g)_!(\Delta_!a) =
\bigl( (f\times g)^{Z}_{YX}\bigr )_!(a).
\]
Part (b) follows from (a) together with $(f\times
f)^{Z}_{YY}=\Delta^Y\circ f$.
Finally, (c) is just the special case of (b) where we take $a=1$ (so $Da=i_X$).
\end{proof}

In fact {\it every\/} morphism in $\cC_E$ can be expressed as an
$RDI$ composition.
This is actually a triviality, but it is 
nevertheless important:

\begin{lemma}
Every element of $\cC_E(X,Y)$ may be written as $R_f\circ D_a\circ
I_g$ for some object $Z$, some maps $f\colon Z\ra Y$, $g\colon Z\ra
X$, and some $a\in E(Z)$.  
\end{lemma}

\begin{proof}
Let $a\in \cC_E(X,Y)=E(Y\times X)$.  Set $Z=Y\times X$.  We claim that
$a = R_{\pi_1}\circ D_a\circ I_{\pi_2}$
in $\cC_E$.  This is immediate from Proposition~\ref{pr:triple}(a).
\end{proof}

Now suppose that we have two maps in $RDI$ form, and that we wish to
compose them; that is, consider a composition of the form
\[ [R_f\circ D_a\circ I_g]\circ [R_{f'}\circ D_{a'}\circ I_{g'}].
\]
There are three rules that allow us to rewrite this in $RDI$ form
once again.  We indicate these schematically as:
\addtocounter{subsection}{1}
\begin{align}
\label{eq:RDI}
D\circ R &\leadsto  R\circ D \qquad \text{[Proposition~\ref{pr:R-D}]} \\
I\circ D &\leadsto  D\circ I \qquad
\notag
\,\text{[Proposition~\ref{pr:R-D}]} \\ \notag
I\circ R &\leadsto  R\circ I \qquad \ \text{[Proposition~\ref{pr:Gysin->Cat}(h)].} 
\end{align}
To use these in our problem, we start by forming the pullback in the
following diagram:
\[ \xymatrix{
&& P\ar[dl]_s\ar[dr]^t \\
& A\ar[dl]_f \ar[dr]^g && A'\ar[dl]_{f'}\ar[dr]^{g'} \\
Y && B && X.
}
\]
Then
\addtocounter{subsection}{1}
\begin{align}
\label{eq:RDI-comp}
R_f\circ D_a \circ I_g\circ R_{f'}\circ D_{a'}\circ I_{g'} &=
R_f\circ D_a \circ R_{s}\circ I_t \circ D_{a'}\circ I_{g'} \\ \notag
&= R_f\circ R_s\circ D_{s^*a} \circ D_{t^*a'} \circ I_t\circ I_{g'} \\
\notag
&= R_{fs}\circ D_{(s^*a)\cdot(t^*a')}\circ I_{g't}.\notag
\end{align}

The above discussion has proven Corollary~\ref{co:intro-RDI} from the
introduction.  We also note that we have proven
Theorem~\ref{th:intro-RDI} along the way:

\begin{proof}[Proof of Theorem~\ref{th:intro-RDI}]
Parts (a) and (b) are Proposition~\ref{pr:R-D}, whereas (c) is
Proposition~\ref{pr:Gysin->Cat}(h).
Part (d) is Proposition~\ref{pr:triple}(b).  
\end{proof}

\subsection{A detailed example of the Burnside functor}

Consider the category $\cC_\cA$, where $\cA$ is the Burnside functor
for $\cC$.  Given the role of $\cA$ as the universal Gysin functor,
it is useful to have a particularly good handle on how to work with
$\cC_\cA$.  Recall that a map in $\cC_\cA$ from $X$ to $Y$ is an
element of $\cA(Y\times X)$, and so is represented by a map $h\colon S\ra
Y\times X$ in $\cC$.  It is often useful to represent this data as a
span, by writing
\[ \xymatrixrowsep{1pc}\xymatrixcolsep{1pc}\xymatrix{ & S \ar[dr]^{\pi_2h}\ar[dl]_{\pi_1h} \\
Y&& X.}
\]

The following list gives a ``dictionary'' for how certain structures
are represented in $\cC_\cA$.

\begin{itemize}
\item[(1)]
$\xymatrixcolsep{1pc}\xymatrixrowsep{0.1pc}\xymatrix{ &X \ar[ddl]_f
\ar[ddr]^\id \\
 &&& = & Rf, \\ A && X}
\qquad\qquad\qquad
\xymatrixcolsep{1pc}\xymatrixrowsep{0.1pc}\xymatrix{ &X \ar[ddl]_\id
\ar[ddr]^g \\
 &&& = & Ig \\ X && B}
$ \\[0.1in]
\item[(2)]
$\xymatrixcolsep{1pc}\xymatrixrowsep{0.1pc}\xymatrix{ &X \ar[ddl]_f
\ar[ddr]^g \\ 
 &&& = & Rf\circ Ig\\
A && B}$ 
\\[0.1in]
\item[(3)]
$
\xymatrixcolsep{1pc}\xymatrixrowsep{0.1pc}\xymatrix{ &X \ar[ddl]_\id
\ar[ddr]^\id \\
 &&& = & i_X \\ X && X}
\qquad\qquad\qquad
\xymatrixcolsep{0.7pc}\xymatrixrowsep{0.1pc}\xymatrix{ &X\times X \ar[ddl]_{\pi_1}
\ar[ddr]^{\pi_2} \\
 &&& = & 1_X \\ X && X}
$
\\[0.1in]
\item[(4)]
$\xymatrixcolsep{0.6pc}\xymatrixrowsep{0.1pc}\xymatrix{ 
&X \ar[ddl]_f \ar[ddr]^g &&&&X'\ar[ddl]_{f'} \ar[ddr]^{g'} &&&&
X\times X'\ar[ddl]_>>>>>{f\times f'}\ar[ddr]^{g\times g'} \\
&&& \tens &&&& =  \\ 
A && B && A' && B' && A\times A' && B\times B'}
$
\\[0.1in]
\item[(5)] unit of $\tens$ is $S=*$ \\[0.1in]
\item[(6)] The adjunction $\Hom(X,F(Y,Z))\ra \Hom(X\tens Y,Z)$ is
\[
\xymatrixcolsep{0.6pc}\xymatrixrowsep{0.1pc}\xymatrix{ 
& W\ar[ddl]_{f\times g} \ar[ddr]^h &&&& W\ar[ddl]_g\ar[ddr]^{h\times f} \\
&&& \longleftrightarrow \\
Y\times Z && X && Z && X\times Y
}
\]
\\[0.1in]
\item[(8)] The identity $\mathbbm{1}_X\colon S\ra F(X,X)$ is
\[
\xymatrixcolsep{0.6pc}\xymatrixrowsep{1pc}\xymatrix{ 
& X \ar[dl]_\Delta \ar[dr]\\
X\times X && {}\phantom{X}{*}\phantom{X}
}
\]
\\
\item[(9)] The evaluation and coevaluation morphisms are
\[
\xymatrixcolsep{0.6pc}\xymatrixrowsep{0.1pc}\xymatrix{ 
 & X\ar[ddl]\ar[ddr]^\Delta &&&&&&& X\ar[ddl]_\Delta \ar[ddr] \\
&&& = & \ev_X, &&&&&& = & \cev_X \\
{}\phantom{X}{*}\phantom{X} && X\times X &&&&& X\times X && {}\phantom{X}{*}\phantom{X}
}
\]
\\
\item[(10)] The transposition $t_{X,Y}\colon X\tens Y\ra Y\tens X$ is
\[
\xymatrixcolsep{0.6pc}\xymatrixrowsep{1pc}\xymatrix{ 
& X\times Y \ar[dl]_t \ar[dr]^\id \\
Y\times X && X\times Y
}
\]
\item[(11)]
Given $a\colon T\ra X$, one has $
\xymatrixcolsep{1pc}\xymatrixrowsep{0.1pc}\xymatrix{ &T \ar[ddl]_a
\ar[ddr]^a \\ 
&&& = & Da. \\
X && X}$ 
\end{itemize}

\subsection{Gysin categories in the Galois setting}
\label{se:Gysin-Galois}

We now add some extra hypotheses to the category $\cC$, all of which
are satisfied in the cases of interest.  First, say that an object $X$
in $\cC$ is \dfn{atomic} if $X\neq \emptyset$ and $X$ 
is not isomorphic to a coproduct
$A\amalg B$ where both $A$ and $B$ are different from the initial
object.  We will assume that

\begin{itemize}
\item If $X$ is atomic and $Y$ and $Z$ are any objects, then the
natural map $\cC(X,Y)\amalg \cC(X,Z)\ra \cC(X,Y\amalg Z)$ is a
bijection.
\item For every atomic object $X$ in $\cC$, the set $\Aut(X)$ is
finite.
\item If $X\neq\emptyset$ then $\cC(X,\emptyset)=\emptyset$.
\end{itemize}

\noindent
If $\cC$ is finitary lextensive and satisfies the above properties, we
will say that $\cC$ is a \dfn{Galoisien} category.

\begin{example}
Let $G$ be a finite group, and let $\cC$ be the category of finite
$G$-sets.  Then $\cC$ is Galoisien, and the atomic objects are the
transitive $G$-sets.  
\end{example}

Let $Y$ be an object of $\cC$, and let $G(Y)=\Aut(Y)$.  There is
an evident map
\[ \coprod_{\sigma \in G(Y)} Y_\sigma \ra Y\times Y \]
(where $Y_\sigma$ denotes a copy of $Y$ labelled by $\sigma$), where
the map $Y_\sigma \ra Y\times Y$ is $\id\times \sigma$.  We say that
$Y$ is \dfn{Galois} if the displayed map is an isomorphism.  To
generalize this somewhat, if $p\colon X\ra Y$ is a map then let
$G(X/Y)=\{\alpha\in \Aut(X)\,|\,p\alpha=p\}$.  Say that $X\ra Y$ is
Galois if the evident map
\[ \coprod_{\sigma \in G(X/Y)} X_\sigma \ra X\times_Y X
\]
is an isomorphism.

The results in the following lemma can be proven by elementary category theory:

\begin{lemma}
\label{le:Galois-lemma}
Suppose that $X$ and $Y$ are atomic.
\begin{enumerate}[(a)]
\item If $Y$ is Galois then $\cC(X,Y)$ is either empty or else it is a
$G(Y)$-torsor.
\item If $Y$ is Galois then every endomorphism of $Y$ is an
isomorphism.
\item If $X$ and $Y$ are Galois and $f,g\colon X\ra Y$, then the
evident map
\[ \coprod_{\{\sigma\in G(X)\,|\, f=g\sigma\}} X_\sigma \lra
\text{\rm pullback}[X\llra{f} Y \llla{g} X]
\]
is an isomorphism.
\item If $X$ and $Y$ are Galois then so is every map $X\ra Y$.
\item Suppose that $Y$ and $Z$ are both Galois, and assume given $f\colon X\ra
Y$ and $g\colon Z\ra Y$.   If there exists a map $u\colon X\ra Z$ such
that $gu=f$, then the evident map
\[ \coprod_{\sigma \in G(Z/Y)} X_\sigma \ra X\times_Y Z
\]
is an isomorphism.
\item If $f\colon X\ra Y$ is a map and $Y$ is Galois, then 
$\coprod_{\sigma\in G(Y)} X_\sigma
\ra X\times Y$ given by $\id\times \sigma f\colon X_\sigma \ra X\times
Y$ is an isomorphism.

\item If $X$ and $Y$ are both Galois and $f\colon X\ra Y$, then for
every $\alpha\in \Aut(X)$ there is a unique $\alpha_f\in \Aut(Y)$ such
that $f\alpha=\alpha_f f$.  Moreover, the map $G(X)\ra G(Y)$ given by
$\alpha\ra \alpha_f$ is a group homomorphism.
\end{enumerate}
\end{lemma}

\begin{proof}
For (a),
suppose that $\cC(X,Y)\neq \emptyset$ and let $f\colon X\ra Y$ be a
map.  
We need to show that the map $G(Y) \ra
\cC(X,Y)$ given by $\sigma \mapsto \sigma f$ is a bijection.
Let $g\colon X\ra Y$ be any map, and consider $f\times g\colon X\ra
Y\times Y$.  Composing with the isomorphism $\coprod_{G(Y)} Y\ra
Y\times Y$, the fact that $X$ is atomic shows that the resulting map
factors through a map $u\colon X\ra Y_\sigma$, for some $\sigma$.  One
then obtains the commutative diagram
\[ \xymatrixcolsep{3pc}\xymatrix{
X\ar[d]_u \ar[dr]^{f\times g} \\
Y_\sigma \ar[r]^-{\id\times \sigma} & Y\times Y
}
\]
which shows that $u=f$ and $g=\sigma u=\sigma f$.  So the action of
$G(Y)$ on $\cC(X,Y)$ is transitive.  

Now suppose that $\alpha,\beta\in G(Y)$ and $\alpha f=\beta f$.  Then
$f\times \alpha f\colon X\ra Y\times Y$ factors through both
$Y_\alpha$ and $Y_\beta$ (under the isomorphism $\coprod_{G(Y)} Y\ra
Y\times Y$).  Therefore it factors through the pullback $Y_\alpha \ra
\coprod_{G(Y)} Y \la Y_\beta$.  But if $\alpha \neq \beta$ then this
pullback is $\emptyset$, by our standing hypotheses that $\cC$ is
finitary lextensive.  Since
the map $f\times \alpha f$ cannot factor through $\emptyset$, this is
a contradiction; so we must have $\alpha=\beta$.

Part (b) is an immediate consequence of (a).
For (c), the pullback in question is isomorphic to the pullback of
\[ Y\llra{\Delta} Y\times Y \llla{f\times g} X\times X.
\]
Use the decomposition $X\times X\iso \coprod_{\sigma \in G(X)}
X_\sigma$ and the fact that pullbacks distribute over finite coproducts to see
that our pullback is isomorphic to
\[ \coprod_{\sigma \in G(X)} \text{\rm{pullback}}[Y\llra{\Delta}
Y\times Y \llla{f\times g\sigma} X].
\]
Next use the decomposition $Y\times Y\iso \coprod_{\alpha\in G(Y)}
Y_\alpha$, together with the fact that $\Delta\colon Y\ra Y\times Y$
factors through the summand $Y_{\id}$.  
Since $X$ is atomic, we deduce that
the pullback inside the above coproduct is either $\emptyset$ (when $f\neq
g\sigma$) or $X$ (when $f=g\sigma$).  This finishes off part (c).

Part (d) is a direct consequence of (c), applied in the case $f=g$.

For (e), the existence of $u$ implies that $X\times_Y Z$ is isomorphic to the
pullback of
\[ X\llra{u} Z \llla{\pi_1} Z\times_Y Z.
\]
Next use that $Z\times_Y Z\iso \coprod_{\sigma\in G(Z/Y)} Z_\sigma$
and use the fact that pullbacks distribute over coproducts.  

Part (f) is a special case of (e).
Finally, the proof of (g) uses the same techniques that have been demonstrated
in the preceding parts: consider $f\times f\alpha \colon X\ra Y\times Y$
and factor this through some $Y_\sigma$.  Details are left to the reader.
\end{proof}

Before proceeding, let us establish some
notation.  If $R$ is a ring and $S$ is a set, then $R\langle S\rangle$
denotes the set of all formal finite sums $\sum r_is_i$ where $r_i\in
R$ and $s_i\in S$.  This is the free left $R$-module with basis $S$.
Similarly, let $\langle S\rangle R$ be the set of all formal finite
sums $\sum s_ir_i$ with $s_i\in S$ and $r_i\in R$.  When $R$ is
commutative these are of course isomorphic $R$-modules, but the
difference in notation will be useful to us below.

When $X$ is  Galois we can now determine the ring $\cC_E(X,X)$
precisely:

\begin{prop}
\label{pr:Galois-endo}
If $X$ is Galois then the map $E(X)[\widetilde{\Aut(X)}]\ra
\cC_E(X,X)$ from Remark~\ref{re:twisted-endo} is an isomorphism of rings.
\end{prop}

\begin{proof}
Since $X$ is Galois, the usual map $\coprod_{\sigma \in G(X)} X\ra X\times
X$ is an isomorphism.   So
\[ B\colon \bigoplus_{\sigma\in G(X)} E(X) \ra E(X\times X)
\]
is an isomorphism, where on component $\sigma$ the map $B$ equals $(\id\times
\sigma)_!\Delta_!$.   
If $a\in E(X)$ then we have a copy of $a$ in the component of the
domain indexed by $\sigma$.  The image of this class in $E(X\times X)$
is precisely 
\begin{align*}
(\id\times \sigma)_!\Delta_!(a)=D_a\circ I_{\sigma}=D_a\circ R_{\sigma^{-1}}.
\end{align*}
This implies that the map $E(X)\langle
\Aut(X)\rangle \ra \cC_E(X,X)$ given by $a.\sigma\mapsto D_aR_\sigma$
is an isomorphism of abelian groups.  We already saw in
Remark~\ref{re:twisted-endo} that it is a ring homomorphism, where we
give the domain the appropriate structure of twisted group ring.
\end{proof}

\begin{remark}
In concrete terms, Proposition~\ref{pr:Galois-endo} says that every
map in $\cC_E(X,X)$ may be uniquely written as a finite sum of terms
$D_a R_\alpha$ where $a\in E(X)$ and $\alpha\in \Aut_\cC(X)$.
Composition is done according to the rule
\[ D_a R_\alpha \circ D_b R_\beta = D_a D_{(\alpha^{-1})^*b}R_\alpha
R_\beta = D_{a\cdot (\alpha^{-1})^*b} R_{\alpha\beta}
\]
where in the first equality we have used Proposition~\ref{pr:R-D}
(together with the fact that $\alpha$ is an isomorphism).
The awkwardness of this formula stems from our
representation of elements of $\cC_E(X,X)$ in the form $D_aR_\alpha$.  As
we have remarked before, it is better to use the $RDI$ system and
represent the elements as $R_\alpha\circ D_a$.  If we do this, then
the composition law is
\[ R_\alpha D_a\circ R_\beta D_b = R_{\alpha\beta} D_{(\beta^*a\cdot b)},
\]  
which is a little simpler.  We will always use this formulation from
now on.  
\end{remark}

We next turn to the case of two objects.
Assume that $f\colon X\ra Y$ is a map in $\cC$, where both $X$ and
$Y$ are assumed to be atomic and Galois.  Our goal is to describe the full
subcategory of $\cC_E$ containing $X$ and $Y$.  If $f$ is an
isomorphism then this problem reduces to the case of one object, which
we handled above.  So let us further 
assume that $f$ is not an
isomorphism.  Note that this implies that there cannot exist a map in
$\cC$ from $Y$ to $X$: if there were such a map, then the post- and pre-composites
with $f$ would be isomorphisms by Lemma~\ref{le:Galois-lemma}(b),
and so $f$ would itself be an isomorphism.

Write
$\Aut(X)=\{\alpha_1,\ldots,\alpha_r\}$ and
$\Aut(Y)=\{\beta_1,\ldots,\beta_s\}$.  
Note that $\cC(X,Y)=\{\beta_1f,\ldots,\beta_s f\}$ by
Lemma~\ref{le:Galois-lemma}(a), and $X\times Y\iso \coprod_{\sigma\in
  \Aut(Y)} X$ by Lemma~\ref{le:Galois-lemma}(f).  Then
$\cC_E(X,Y)=E(Y\times X)\iso \oplus_{\Aut(Y)} E(X)$, and one can check
that the isomorphism is the one that represents each map in
$\cC_E(X,Y)$
as a sum of maps $R_{\beta_i}D_a$ where $a\in E(X)$.  A similar
analysis works for $\cC_E(Y,X)$, and so
the full subcategory of $\cC_E$ containing $X$ and $Y$ may be depicted
as follows:

\[ \xymatrixcolsep{8pc}\xymatrix{
X \ar@/^5ex/[r]^{\langle R_{\beta_1f},\ldots,R_{\beta_sf}\rangle E(X)}
\ar@(ul,dl)[]_{\langle R_{\alpha_1},\ldots,R_{\alpha_r}\rangle E(X)}
& Y
\ar@(ur,dr)[]^{\langle R_{\beta_1},\ldots,R_{\beta_s}\rangle E(Y)} 
\ar@/^5ex/[l]^{E(X)\langle I_{\beta_1f},\ldots,I_{\beta_s f}\rangle}
}
\]
The labels on the arrows depict the abelian group of maps in $\cC_E$; e.g., the
label on the arrow from $X$ to $Y$ depicts $\cC_E(X,Y)$.  The diagram
indicates that every map
from $X$ to $Y$ may be uniquely written as a sum of terms
$R_{\beta_if}D_{a_i}$ where $a_i\in E(X)$ (and similarly for other
choices of domain and range).  

Compositions of maps are determined via the $RDI$ rules outlined in
(\ref{eq:RDI}).  Here are some examples:

\begin{enumerate}[(1)]
\item {}[$X\ra X\ra Y$ compositions.] Here one uses
\[ 
R_{\beta_i f}D_a\circ R_{\alpha_j}D_b= R_{\beta_i f}
R_{\alpha_j}D_{\alpha_j^*(a)}D_b = R_{\beta_i
  f\alpha_j}D_{(\alpha_j^*a)b}.
\]
\item {}[$Y\ra Y\ra X$ compositions.]  Here one uses
that $\beta_i$ is invertible and so we have $R_{\beta_i}=I_{\beta_i}^{-1}$:
\begin{align*}
D_aI_{\beta_j f}\circ R_{\beta_i}D_u= D_aI_{\beta_j f}\circ I_{\beta_i}^{-1}
D_u = D_a I_{\beta_j f}I_{\beta_i^{-1}}D_u & = D_a
I_{\beta_i^{-1}\beta_j f} D_u\\
& =D_{(a\cdot (\beta_i^{-1}\beta_jf)^*(u))}
I_{\beta_i^{-1}\beta_j f}.
\end{align*}
\item {}[$Y\ra X\ra Y$ compositions.]  In this case
we consider
\begin{align*}
 R_{\beta_j f}D_a\circ D_bI_{\beta_i f}= R_{\beta_j} \circ R_f \circ
D_{ab} \circ I_f\circ I_{\beta_i} & = R_{\beta_j}\circ D_{f_!(ab)} \circ
I_{\beta_i} \\
& = R_{\beta_j}\circ I_{\beta_i} \circ D_{(\beta_i^{-1})^*(f_!(ab))}\\
& = R_{\beta_j}\circ R_{\beta_i^{-1}} \circ
D_{(\beta_i^{-1})^*(f_!(ab))}\\
&= R_{\beta_j\beta_i^{-1}}D_{(\beta_i^{-1})^*(f_!(ab))}.
\end{align*}
In the second equality we have used Proposition~\ref{pr:triple}(b)
 and in the third equality we
have used Proposition~\ref{pr:R-D} (which applies because $\beta_i$ is invertible).
\item {}[$X\ra Y\ra X$ compositions.]  Let $T=\{\sigma\in G(X)\,|\,
\beta_jf=\beta_if\sigma\}$.  Observe that by Lemma~\ref{le:Galois-lemma}(c) 
since $X$ and $Y$ are Galois
we have a pullback
diagram
\[ \xymatrix{
\coprod_{\sigma\in T}
X_\sigma \ar[r]\ar[d] & X \ar[d]^{\beta_i f} \\
X\ar[r]^{\beta_j f} & Y
}
\]
where the vertical map $X_\sigma \ra X$ is the identity and the
horizontal map $X_\sigma\ra X$ is $\sigma$.  We then write
\begin{align*}
 D_{a'}I_{\beta_j f}\circ R_{\beta_i f}D_a=\sum_\sigma D_{a'}
I_{\sigma} D_a & = \sum_\sigma I_{\sigma} D_{(\sigma^{-1})^*(a')}D_a\\
& = 
\sum_\sigma R_{\sigma^{-1}} D_{((\sigma^{-1})^*(a')\cdot a )}.
\end{align*}
The second equality is by Proposition~\ref{pr:R-D}, using that
$\sigma$ is an isomorphism.

\item {}[Remaining cases.]  The cases that have not been treated so far are all very
similar to (1) or (2).
\end{enumerate}

As the reader can see from the above analysis, a complete description
of the maps between Galois objects is relatively simple.  But the
description of compositions becomes unwieldy, although in practice it
is a purely mechanical process to work out any given composition.  


\section{Grothendieck-Witt categories over a field}
\label{se:GW-field}

Let $k$ be a field of characteristic not equal to $2$, 
and consider the Grothendieck-Witt category
$\GWC(k)$ over $k$.  

Let $\fEt/k$ be the full subcategory of $\Aff/\Spec k$ consisting of
the objects $\Spec E$ where $k\ra E$ is finite \'etale.  Let
$\cA_{\fEt}$ be the Burnside Gysin functor, and let $\chi\colon
\cA_{\fEt}\ra \GW$ be the natural transformation from Proposition~
\ref{pr:Gysin-universal}.  

The following result is essentially \cite[Appendix B, Theorem 3.1]{Dr}.  We
include the proof for completeness.  For the proof, recall 
that if $a\in E$ then $\langle
a\rangle$ denotes the quadratic space $(E,b_a)$ where $b_a(x,y)=axy$,
and $\langle a,b\rangle=\langle a\rangle\oplus \langle b\rangle$.

\begin{prop}
\label{pr:A->GW}
For any finite separable field extension $k\ra E$, the map $\chi\colon
\cA_{\fEt}(E)\ra
\GW(E)$ is surjective.
\end{prop}

\begin{proof}
Recall that $\GW(E)$ is generated as an abelian group by the classes
$\langle a\rangle$ for $a\in E^*$.  We will show that each of these
classes is in the image of $\chi$. 

If $a$ is not a square in $E$ then consider the field extension
$E_a=E[x]/(x^2-a)$.  Then $E_a$ is a separable field extension of $E$,
and $\chi(E_a)$ is simply $E_a$ (regarded as an
$E$-vector space) equipped with
the trace form.  An easy computation shows this is
isomorphic to $\langle 2,2a\rangle=\langle 2\rangle + \langle
2a\rangle$.    So we have $\langle 2\rangle + \langle
2a\rangle=\chi(E_a)$.

We claim that $\langle 2\rangle \in \im \chi$.  If $2$ is a square in
$E$ then this is clear, since $\langle 2\rangle=\langle 1\rangle$.  
If $2$ is not a square in $E$ then we may apply the above analysis
with $a$ replaced by $2$ to find that $\langle 2\rangle + \langle
4\rangle \in \im \chi$.  Since $\langle 4\rangle=\langle 1\rangle \in
\im \chi$,
 we again have $\langle 2\rangle \in \im \chi$.

At this point we know that $\langle 2\rangle +\langle 2a\rangle\in \im
\chi$ and $\langle 2\rangle \in \im \chi$, and so $\langle 2a\rangle
\in \im \chi$.  But then $\langle 4a\rangle=\langle 2\rangle \cdot
\langle 2a\rangle \in \im \chi$.  Since $\langle 4a\rangle =\langle
a\rangle$, we are done.
\end{proof}

\begin{example}
The map $\chi$ is usually not an isomorphism.  To see this in one
example, let $k=\F_p$ where $p$ is odd.  Then $\Aet(k)$ is a free
abelian group on a countably-infinite set of generators, whereas
$\GW(k)\iso \Z\oplus \Z/2$.  In general, it would be interesting to
have a set of generators for the kernel of $\Aet(k)\ra \GW(k)$
together with some kind of geometric source for them.  See
Example~\ref{ex:euler-ffield} below.
\end{example}

\begin{remark}
If $f\colon R\ra S$ is a sheerly separable map of rings, we have the induced maps
$f_*\colon \GW(R)\ra \GW(S)$ and $f^!\colon \GW(S)\ra \GW(R)$ from
Section~\ref{se:GWcat}.   However, for most purposes it
 is more convenient to use the
geometric setting of affine schemes: there we would write $f^*\colon
\GW(\Spec R)\ra \GW(\Spec S)$ and $f_!\colon \GW(\Spec S)\ra \GW(\Spec
R)$.  The disadvantage here is that it becomes tedious to write
$\Spec$ repeatedly.  We will tend to mix the two notations and write
$f^*\colon \GW(R)\ra \GW(S)$ and $f_!\colon \GW(S)\ra \GW(R)$.  In
effect, this is basically  just dropping the ``$\Spec$'' and letting it
be understood.  In practice there is never any confusion here.
\end{remark}

Our goal is to be able to analyze pieces of the categories $\GWC(k)$ for some
explicit choices of $k$.  Galois theory gives an equivalence of
categories between sheerly separable extensions of $k$ and 
continuous $\Gal(k^{sep}/k)$-sets, and this is a useful tool to
exploit.

Fix a finite-dimensional Galois extension $L/k$, and set
$G=\Gal(L/k)$.  Say that a separable $k$-algebra $A$ is
\mdfn{$L$-constructible} if it is isomorphic to a product $\prod_i A_i$
where each $A_i$ is an algebraic field extension of $k$ that admits an
embedding into $L$.  
For each finite $G$-set $S$, let $\cF(S,L)$ be the set of $G$-maps from $S$ to $L$, with ring structure
given by pointwise addition and multiplication.  Clearly 
$\cF(G/H,L)\iso L^H$ and $\cF(S\amalg T, L)\iso \cF(S,L)\times
\cF(T,L)$,
 hence each $\cF(S,L)$ is
$L$-constructible.  In the opposite direction, given a sheerly separable
$k$-algebra $A$ the set of $k$-algebra maps $\alg{k}(A,L)$ inherits an
action of $G$.  
Galois theory says that we have an equivalence of categories
\[ \fGset \adjoint \fEt_k^{L-con}
\]
where the upper arrow is $S\mapsto \cF(S,L)$ and the lower arrow is $\Spec A\mapsto
\alg{k}(A,L)$.

The Grothendieck-Witt functor on $\fEt_k$ restricts, via the above
Galois equivalence, to a Gysin functor on finite $G$-sets.  Let us
write
\[ \GW_L(S)=\GW(\cF(S,L))
\]
for this restricted Gysin functor.  Clearly the correspondence
category
$\fGset_{(\GW_L)}$ 
is the full subcategory of $(\fEt_k)_{\GW}$ whose
objects are the $L$-constructible $k$-algebras.  

The universality of the Burnside functor gives a natural
transformation $\cA_G \ra \GW_L$, and therefore a functor between
correspondence categories
\[ \fGset_{(\cA_G)} \ra \fGset_{(\GW_L)}.
\]
Putting everything together, we have constructed a functor from the
Burnside category of $G$ to the Grothendieck-Witt category over $k$.  

We now look at several examples:

\begin{example}
The category $\GWC(\R)$ has two objects: $\Spec \R$ and $\Spec \C$.  Let
$\pi\colon \Spec \C \ra \Spec \R$ be the unique map, and $\sigma\colon
\Spec \C\ra \Spec \C$ be the nontrivial automorphism.  
Since $\GW(\C)=\Z$ and $\GW(\R)=\Z\langle  \langle 1\rangle, \langle
-1\rangle \rangle$, the category $\GWC(\R)$ is readily computed to be
as shown in the diagram below.  One only needs check that $I_\pi\circ
R_\pi=1+\sigma$ and $R_\pi\circ I_\pi=\langle 1\rangle + \langle
-1\rangle$.  

Similarly, the Burnside category for $\Z/2$ has two objects: $*$ and $\Z/2$.
We write $\pi\colon \Z/2\ra *$ and $\sigma\colon \Z/2\ra \Z/2$ for the
evident maps.  Then $\cA_{\Z/2}(\Z/2)=\Z$ and $\cA_{\Z/2}(*)=\Z\langle
[*],[\Z/2]\rangle$.  Here one computes that $I_\pi\circ
R_\pi=1+\sigma$ and $R_\pi\circ I_\pi=[\Z/2]$.  

\vspace{0.2in}

\[ \xymatrix{
\Z/2 \ar@(ul,ur)^{\Z\langle 1,\sigma\rangle} \ar@/^3ex/[d]^{\Z\langle R\pi\rangle} &&&& \Spec \C
\ar@/^3ex/[d]^{\Z\langle R\pi\rangle} \ar@(ul,ur)^{\Z\langle 1,\sigma\rangle} \\
{*} \ar@/^3ex/[u]^{\Z\langle I\pi\rangle} \ar@(dl,dr)_{\Z\langle
  [*],[\Z/2]\rangle} 
&&&& \Spec \R. \ar@/^3ex/[u]^{\Z\langle I\pi\rangle}
\ar@(dl,dr)_{\Z\langle\, \langle 1\rangle,\langle -1\rangle\,\rangle}
}
\]
The map from the Burnside category to the Grothendieck-Witt category
has the evident behavior (in particular, it sends $[\Z/2]$ to
$\langle 1\rangle +\langle -1\rangle$), and by inspection is an isomorphism.
\end{example}

Before considering our next example we need to recall some facts about
finite fields.  If $F$ is a finite field of odd characteristic 
then $F^\times$ is cyclic of even order and
so $(F^\times)/(F^\times)^2=\Z/2$.  Thus when we partition $F^\times$
into the squares and the non-squares, any two non-squares are
equivalent: if $a$ and $b$ are non-squares then $a=\lambda^2 b$ for
some $\lambda$.  
A little work shows when $\chara(F)\neq 2$ that $\GW(F)$ is generated by
$\langle 1\rangle$ and $\langle g\rangle$, where $g\in F^\times$ is
any choice of non-square.  Moreover, $2\langle g\rangle=2\langle
1\rangle$ and $\GW(F)\iso \Z\oplus \Z/2$ with corresponding generators
$\langle 1\rangle$ and $\langle g\rangle-\langle 1\rangle$.  See
\cite{S} or \cite[Appendix A]{D1} for details.  
It is useful to write $\alpha=\qf{g}-\qf{1}$.

Note that the calculation of $\GW(F)$ gives a classification of all
non-degenerate quadratic spaces over $F$: in each dimension there are
exactly two, namely $n\langle 1\rangle$ and $(n-1)\langle 1\rangle +
\langle g\rangle=n\qf{1}+\alpha$.  The discriminant of the form,
regarded as an element of $F^\times/(F^\times)^2$, distinguishes the
two isomorphism types.

The following lemma calculates the behavior of the Grothendieck-Witt
group under a quadratic extension.

\begin{lemma}
\label{le:GW1}
  Let $q=p^e$ where $p$ is an odd prime.  Fix a non-square $g\in
  \F_q$, and fix a non-square $h\in \F_{q^2}$.  If $j\colon \F_q\inc
  \F_{q^2}$ is a fixed embedding then the pullback and pushforward
  maps for $\GW(\blank)$ are given by the formulas
\[ j^*(\langle 1\rangle)=j^*(\qf{g})=\qf{1}, \qquad 
j_!(\qf{1})=\qf{1}+\qf{g}, \quad j_!(\qf{h})=2\qf{1}.
\]
Every automorphism of $\F_q$ induces the identity on $\GW(\F_q)$ (both
via pullback and pushforward).  
\end{lemma}

\begin{proof}
First note that if $\alpha$ is an automorphism of $\F_q$ then $\alpha$
  preserves the property of being a square or non-square; consequently,
  $\alpha^*$ is the identity since $\alpha^*(\qf{g})=\qf{g}$.  Since $\alpha_!$
  is the inverse of $\alpha^*$ (Lemma~\ref{le:iso-shriek}), this is also the
  identity.  So we have verified the last sentence of the lemma.

  Observe that $\F_{q^2}$ may be identified with the extension
  $\F_q[x]/(x^2-g)$, and we may assume that $j$ is the evident
  inclusion of $\F_q$ (using the previous paragraph).  Since $g=x^2$
  in $\F_{q^2}$ we have $j^*(\qf{g})=\qf{1}$.

To compute $j_!(\qf{1})$ we must analyze the trace form on
$\F_{q^2}$.  This is represented by the $2\times 2$ matrix
\[ 
\begin{bmatrix}
\tr(1) & \tr(x) \\
\tr(x) & \tr(x^2)
\end{bmatrix} =
\begin{bmatrix}
2 & 0 \\
0 & 2g
\end{bmatrix}.
\]
The discriminant is $4g$, which is equivalent to $g$ modulo squares.
So $j_!(\qf{1})=\qf{1}+\qf{g}$.  

The above work readily generalizes to compute $j_!(\qf{a+bx})$ for any
$a,b\in \F_q$.  This form is represented by the matrix
\[ 
\begin{bmatrix}
\tr(a+bx) & \tr(ax+bx^2) \\
\tr(ax+bx^2) & \tr(ax^2+bx^3)
\end{bmatrix} =
\begin{bmatrix}
2a & 2bg \\
2bg & 2ag
\end{bmatrix}.
\]
The discriminant is $4a^2g-4b^2g^2=4g(a^2-b^2g)$, and so 
$j_!(\qf{a+bx})=\qf{1}+\qf{g(a^2-b^2g)}$.  

In a finite field every element can be written as a sum of two squares
\cite[Lemma 2.3.7]{S}, so we can write $g^{-1}=b^2+r^2$ for some
$b,r\in \F_q$.  Neither $b$ nor $r$ is zero, since $g$ is not a
square.  Then
\[
j_!(\qf{1+bx})=\qf{1}+\qf{g(1-b^2g)}=\qf{1}+\qf{g(r^2g)}=\qf{1}+\qf{1}=2\qf{1}.
\]
Hence $\qf{1+bx}\neq \qf{1}$ (since their images under $j_!$ are
different), and so $1+bx$ is a non-square class;
i.e. $\qf{1+bx}=\qf{h}$ in $\GW(\F_{q^2})$.  So we have in fact proven
that $j_!(\qf{h})=2\qf{1}$.
\end{proof}

\begin{prop}
\label{pr:GW-ffield}
Let $q$ be a power of an odd prime, and consider a field extension
$j\colon \F_q\inc \F_{q^e}$.  Let $g$ and $g'$ be non-squares in
$\F_q$ and $\F_{q^e}$, respectively.
Then the induced maps $j^*$ and $j_!$
are given by
\[ j^*(\qf{1})=\qf{1},\qquad
\j^*(\qf{g})=\begin{cases}
\qf{g'} & \text{if $e$ is odd,}\\
\qf{1} & \text{if $e$ is even,}
\end{cases}
\]
\[ j_!(\qf{1})=
\begin{cases}
e\qf{1} & \text{$e$ odd,}\\
(e-1)\qf{1}+\qf{g}  & \text{$e$  even,}
\end{cases}
\qquad
j_!(\qf{g})=
\begin{cases}
(e-1)\qf{1} + \qf{g} & \text{$e$ odd,}\\
e\qf{1} & \text{$e$ even.}
\end{cases}
\]
These formulas can also be written as:
\[ j^*(\qf{1})=\qf{1}, \qquad j^*(\alpha)=\begin{cases} 
\alpha & \text{$e$ odd},\\
0 & \text{$e$ even},
\end{cases}
\]
\[
j_!(\qf{1})=\begin{cases}
e\qf{1} & \text{$e$ odd} \\
e\qf{1}+\alpha & \text{$e$ even}
\end{cases}, \qquad \qquad
j_!(\alpha)=\alpha.
\]
\end{prop}

\begin{proof}
The statement about $j^*$ is immediate: the extension $\F_{q^e}$
contains a square root of $g$ if and only if it contains $\F_{q^2}$,
which happens precisely when $e$ is even.

To compute $j_!(\qf{1})$ it suffices to analyze the discriminant of
the trace form on $\F_{q^e}$.  A classical computation says this
coincides with the discriminant of the minimal polynomial of any
primitive element for the extension $\F_{q^e}/\F_q$.  
If $r_1,\ldots,r_e$ are the roots of this minimal polynomial, then
this discriminant is $\Delta=Q^2$ where
\[ Q=\prod_{i<j} (r_i-r_j).
\]
If the roots are indexed appropriately then the Galois group of
$\F_{q^e}/\F_q$ acts by cyclic permutation.  It follows that $Q$ is
invariant under the Galois action if and only if $e$ is odd. 
So we see that $\Delta$ is a square in $\F_{q}$ if and only if $e$ is
odd.  The former condition is equivalent to
$j_!(\qf{1})=e\qf{1}$.

Finally, we analyze $j_!(\qf{g})$.  When $e$ is odd this is easy, as
we can write
\[ j_!(\qf{g})=j_!(j^*(\qf{g})\cdot 1)=\qf{g}\cdot
j_!(\qf{1})=\qf{g}\cdot e\qf{1}=e\qf{g}=(e-1)\qf{1}+\qf{g}.
\]
When $e$ is even the pushforward $\GW(\F_{q^{e}})\ra
\GW(\F_{q^{e/2}})$ sends $\langle g\rangle$ to $2\langle 1\rangle$
by Lemma~\ref{le:GW1}.  It follows that $j_!(\langle g\rangle)$ is a
multiple of $2$, and of course it also has rank $e$.  The only such
element of $\GW(\F_q)$ is $e\langle 1\rangle$.  
\end{proof}

\begin{example}[The Euler characteristic of a finite field extension]
\label{ex:euler-ffield}
Our goal is to explicitly compute the map $\chi\colon \Aet(\F_q)\ra
\GW(\F_q)\iso \Z\oplus \Z/2$.  
Given a finite field extension $j\colon \F_q\inc \F_{q^e}$, 
the Euler characteristic is 
another name for $j_!(1)$.
Using Proposition~\ref{pr:GW-ffield},
this is equal to $\chi(\F_{q^e})=e\qf{1}+\epsilon_e \alpha \in \GW(\F_q)$ where
\[ \epsilon_e=\begin{cases}
0 & \text{if $e$ is odd,}\\
1 & \text{if $e$ is even.}
\end{cases}
\]
It is an amusing exercise to use the above computation to check
the
multiplicativity formula
\[ \chi(\F_{q^e}\tens_{\F_q} \F_{q^f})=\chi(\F_{q^e})\cdot \chi(\F_{q^f}),
\]
which is the analog in the present context of the topological formula 
$\chi(X\times Y)=\chi(X)\times
\chi(Y)$.  

We can use the above computation to give generators for the kernel of
$\chi\colon \Aet(\F_q)\ra \GW(\F_q)$.  If we set $E_n=[\F_{q^n}]$
then by inspection a complete set of relations is
\[ E_{n+3}=E_{n+2}+E_{n+1}-E_{n}\ \  (n\geq 1), \quad 2E_2=E_1+E_3, \quad
  E_3=3E_1.
\]
It would be interesting to find an explicit  geometric explanation for these
relations.  For example, one might try to produce a degree $4$ \'etale map
$f\colon X\ra
Y$ of $\F_q$-schemes 
where $Y$ is $\A^1$-connected and where one fiber of $f$ is $\Spec
\F_{q^2}\amalg \Spec \F_{q^2}$ and another fiber is $\Spec \F_q\amalg
\Spec \F_{q^3}$.  
\end{example}

\begin{example}
We next explore a small piece of $\GWC(\F_p)$, where $p$ is odd.  
Specifically, consider
the full subcategory whose objects are $\Spec \F_{q}$ for $q=p^{2^i}$
and $0\leq
i\leq 3$.  Set $G=\Gal(\F_{p^8}/\F_p)=\Z/8$.  
 Let $g_{2^i}$ denote some specific
choice of non-square element in $\F_{p^{2^i}}$, and write
$\alpha_{2^i}=\langle g_{2^i}\rangle -\langle 1\rangle$.    
Also write $J_{2^i}=\GW(\F_{p^{2^i}})$; this is isomorphic to
$\Z\oplus \Z/2$ with corresponding generators $1$ and $\alpha_{2^i}$,
subject to the multiplicative relation
$\alpha_{2^i}^2=-2\alpha_{2^i}=0$.  Finally, let $\sigma$ always
denote the Frobenius $x\mapsto x^p$ and fix specific embeddings
$j_{2^i}\colon\F_{q^{2^i}}\inc \F_{q^{2^{i+1}}}$ and their induced
maps $\pi_{2^i}\colon \Spec \F_{q^{2^{i+1}}}\ra \Spec \F_{q^i}$.  

The following diagrams show the Burnside category for $\Z/8$ as well
as the relevant piece of $\GWC(\F_p)$.  Recall that if $R$ is a ring
and $S$ is a set then we write $R\langle S \rangle$ and $\langle
S\rangle R$ for the sets of
finite sums $\sum r_is_i$ and $\sum s_ir_i$ where $r_i\in R$, $s_i\in
S$. 
We let
$A_{2^i}=\cA_{\Z/8}(\Z/2^i)$, the Grothendieck ring of $\Z/8$-sets
over $\Z/2^i$.  So $A_{2^i}=\Z\langle
[\Z/2^i],[\Z/2^{i+1}],\ldots,[\Z/8]\rangle$.  In the $\Z/8$-set
context we let $\sigma$ always denote the map $x\mapsto x+1$.

\usetikzlibrary{matrix}
\begin{tikzpicture}
\matrix (m) [matrix of math nodes, row sep=2em, column sep=16em,
minimum width=2em]
{
\Z/8 & \F_{p^8} \\
\Z/4 & \F_{p^4} \\
\Z/2 & \F_{p^2} \\
{\phantom{A}\!\!\!\!\!*} & \F_{p} \\};
\path[-stealth]
(m-1-1) edge  [bend left] node[auto] 
{$\scriptstyle{\langle R\pi_4,\ldots,R\pi_4\sigma^3\rangle A_8}$}
(m-2-1)
(m-2-1) edge [bend left] node[auto]
{$\scriptstyle{
A_8\langle I\pi_4,\ldots,I(\sigma^3\pi_4)\rangle }$} (m-1-1)
(m-1-1) edge  [loop left] node[auto] 
{$
\scriptstyle{\langle 1,\sigma,\ldots,\sigma^7\rangle A_8}
$}
(m-1-1)
(m-2-1) edge  [bend left] node[auto] 
{$\scriptstyle{\langle R\pi_2,R\pi_2\sigma\rangle A_4
}$}
(m-3-1)
(m-3-1) edge [bend left] node[auto]
{$\scriptstyle{A_4\langle I\pi_2,I(\sigma\pi_2)\rangle
}$} (m-2-1)
(m-2-1) edge  [loop left] node[auto] 
{$
\scriptstyle{\langle 1,\sigma,\sigma^2,\sigma^3\rangle A_4
}
$}
(m-2-1)
(m-3-1) edge  [bend left] node[auto] 
{$\scriptstyle{\langle R\pi_1\rangle A_2
}$}
(m-4-1)
(m-4-1) edge [bend left] node[auto]
{$\scriptstyle{A_2\langle I\pi_1\rangle
}$} (m-3-1)
(m-3-1) edge  [loop left] node[auto] 
{$
\scriptstyle{\langle 1,\sigma\rangle A_2
}
$}
(m-3-1)
(m-1-2) edge  [bend left] node[auto] 
{$\scriptstyle{\langle R\pi_4,\ldots,R\pi_4\sigma^3\rangle J_8}$}
(m-2-2)
(m-2-2) edge [bend left] node[auto]
{$\scriptstyle{
J_8\langle I\pi_4,\ldots,I(\sigma^3\pi_4)\rangle }$} (m-1-2)
(m-1-2) edge  [loop left] node[auto] 
{$
\scriptstyle{\langle 1,\sigma,\ldots,\sigma^7\rangle J_8}
$}
(m-1-2)
(m-2-2) edge  [bend left] node[auto] 
{$\scriptstyle{\langle R\pi_2,R\pi_2\sigma\rangle J_4
}$}
(m-3-2)
(m-3-2) edge [bend left] node[auto]
{$\scriptstyle{J_4\langle I\pi_2,I(\sigma\pi_2)\rangle
}$} (m-2-2)
(m-2-2) edge  [loop left] node[auto] 
{$
\scriptstyle{\langle 1,\sigma,\sigma^2,\sigma^3\rangle J_4
}
$}
(m-2-2)
(m-3-2) edge  [bend left] node[auto] 
{$\scriptstyle{\langle R\pi_1\rangle J_2
}$}
(m-4-2)
(m-4-2) edge [bend left] node[auto]
{$\scriptstyle{J_2\langle I\pi_1\rangle
}$} (m-3-2)
(m-3-2) edge  [loop left] node[auto] 
{$
\scriptstyle{\langle 1,\sigma\rangle J_2
}
$}
(m-3-2);
\end{tikzpicture}
\end{example}

Notice that we have written $\sigma^i$ instead of $R\sigma^i$.  Also, 
note that
$\sigma$ acts trivially on each $J_n$ by
Lemma~\ref{le:GW1} and so the
endomorphism ring of $\F_{p^{n}}$ is the group ring $J_{n}[\Z/n]$.  
The analogous remark holds in the Burnside category.  Finally, note
that while the two categories clearly have very similar forms, the map
between them is not an isomorphism because $A_i\not\iso J_i$.  

Below we list the main relations in $\GWC(\F_p)$.  Recall that
$\alpha_n\in J_n$ is the unique element of order $2$.  We simplify
$D_a$ to just $a$, for $a\in J_n$.  

\begin{align*}
& R\pi_n\circ I\pi_n = \qf{2}+\alpha_n \in J_n\qquad \qquad  & I\pi_n\circ
R\pi_n=1+\sigma^n \\
& \alpha_n\circ R\pi_n=0 \qquad\qquad & I\pi_n\circ \alpha_n=0  \\
& R\pi_n \circ \alpha_{n+1}\circ I\pi_n=\alpha_{n}
\end{align*}
We leave the reader to derive these, as they are simple consequences
of using the $RDI$ rules together with the computations in
Proposition~\ref{pr:GW-ffield}.
Coupled with the obvious relations that come from the category of
fields, e.g. $R\pi_n \circ \sigma^n=R\pi_n$, the above relations allow
one to work out all compositions in $\GWC(\F_p)$.  

\medskip

\begin{example}
We describe one last example, this time concerning non-Galois
extensions.  Most of the
details will be left to the reader.  Write
$E_2=\Q(\sqrt[3]{2})$, $E_\mu=\Q(\mu_3)$ (the cyclotomic field), and
$E_{2,\mu}=\Q(\sqrt[3]{2},\mu_2)$.  Note that $[E_2:\Q]=3$,
$[E_3:\Q=2]$, and $[E_{2,\mu}:\Q=6]$.  The extensions $E_\mu/\Q$ and
$E_{2,\mu}/E_\mu$ are Galois, but $E_2/\Q$ is not.  Let $\pi_i$, $i\in
\{0,1,2,3\}$, be the maps of schemes induced by the evident inclusions
of fields:  
\[ \xymatrix{
\Spec E_2 \ar[d]_{\pi_0} & \Spec E_{2,\mu} \ar[l]_{\pi_1}\ar[d]^{\pi_3} \\
\Spec \Q & \Spec E_\mu.\ar[l]_{\pi_2}
}
\]
Finally, write $\GW_\mu=\GW(E_\mu)$, and so forth.  

Computing in the Grothendieck-Witt category $\GWC(\Q)$, maps between
$E_\mu$ and $\Q$, or between $E_{2,\mu}$ and $E_\mu$, are handled
exactly as the general case discussed at the end of Section~\ref{se:struc}.
For maps from $E_2$ to $E_\mu$, as an abelian group this is
$\GW_{2,\mu}$ since $E_2\tens_\Q E_\mu=E_{2,\mu}$.  A little thought
shows that the maps are all of the form $R\pi_3\circ Da_{2,\mu}\circ
I\pi_1$, where $a_{2,\mu}\in \GW_{2,\mu}$. 

To compute maps from $E_2$ to itself, we start with $E_2\tens_\Q
E_2\iso E_2\times E_{2,\mu}$.   As an abelian group we then have
$\GWC(\Q)(E_2,E_2)=\GW_2 \oplus \GW_{2,\mu}$.  The two summands
correpond to elements $Da_2$ for $a_2\in \GW_2$ and $R\pi_1\circ
Da_{2,\mu}\circ I\pi_1$ where $a_{2,\mu}\in \GW_{2,\mu}$.  The ring
structure is determined by the formulas
\begin{align*}
   Da_2\circ Db_2&=D(a_2b_2),\\
   Da_2\circ (R\pi_1 \circ D{a_2,\mu}\circ I\pi_1) &= R\pi_1\circ
  D(\pi_1^*(a_2)\cdot a_{2,\mu})\circ I\pi_1\\
   (R\pi_1 \circ Da_{2,\mu}\circ I\pi_1)\circ Da_2 &=
R\pi_1 \circ D(a_{2,\mu}\cdot \pi_1^*(a_2))\circ I\pi_1 \\
 (R\pi_1 \circ Da_{2,\mu} \circ I\pi_1)\circ
(R\pi_1 \circ Db_{2,\mu} \circ I\pi_1)& =
[R\pi_1\circ
D(a_{2,\mu}b_{2,\mu})\circ I\pi_1] +\\
&
\qquad [R\pi_1\circ
D(\sigma^*(a_{2,\mu})b_{2,\mu})\circ I\pi_1]. 
\end{align*}
These equations all follow from the rules in
Theorem~\ref{th:intro-RDI}.

To get a sense of the above computation, let us generalize things just
a bit.  Let $f\colon R\ra S$ be a homomorphism of commutative rings,
and let $\sigma \colon S\ra S$ be an automorphism such that
$\sigma^2=\id$ and $\sigma f=f$.  Define a product on $R\times S$ by
\[ (r,s)\cdot (r',s')=(rr',(f^*r)s'+s(f^*r')+ss'+\sigma(s)s').
\]
Check by brute force that this makes the abelian group $R\times S$
into a ring.  
Let $\alpha$ be the unique $E_2$-linear automorphism of $E_{2,\mu}$
that has  order $2$.
Applying the above construction to $\pi_1^*\colon \GW(E_2)\ra
\GW(E_{2,\mu})$, where $\sigma=\alpha^*$, yields the endomorphism ring
of $E_2$ in the Grothendieck-Witt category $\GWC(\Q)$.
\end{example}


\appendix

\section{Symmetric monoidal categories and duality}
\label{se:dual}

In this section we review some elements from the theory of closed, symmetric
monoidal categories.  Then we recall the notion of a dualizable object,
as well as some standard properties.  

\subsection{Basic conventions}
Let $(\cC,\tens,S,F(\blank,\blank))$ be a closed symmetric monoidal
category.  This means $\tens$ is the monoidal structure, $S$ is the
unit, and $X,Y\mapsto F(X,Y)$ is the cotensor.  

In this setting there are evident evaluation
maps
\[ F(A,B)\tens A\ra B 
\]
defined as the adjoint to the identity on $F(A,B)$.  Likewise, there
are certain canonical maps
\[ F(X,S)\tens Y \ra F(X,Y) \qquad\text{and}\qquad F(A,B)\tens
F(X,Y)\ra F(A\tens X,B\tens Y)
\]
defined to be the adjoints of evident compositions involving symmetry
isomorphisms and
evaluations.  In general, we will use $\psi$ to denote any such
canonical map that arises in a general closed symmetric monoidal
category.  It should always be clear from context exactly what map we mean.

There is one special case where it is useful to have a distinguished
name, rather than just the generic ``$\psi$''.
For any object $X$ in a closed symmetric monoidal category, set
$X^*=F(X,S)$.  Then we let
$\ev_X\colon X^*\tens X\ra S$ be the adjoint of the identity map
$X^*\ra F(X,S)$.

\subsection{Dualizable objects}

The theory of dualizable objects goes back to Dold and Puppe
\cite{DP}, 
but in modern
times has been used extensively by May and his collaborators (see
\cite{LMS} and \cite{Ma1}, for example).

\begin{defn}
\label{de:dual}
An object $X$ in a symmetric monoidal category 
is called \dfn{dualizable} if there is another object
$Y$ together with maps
\[ \eta\colon S\ra X\tens Y, \qquad \epsilon\colon Y\tens X\ra S
\]
such that the composite
\[ 
\xymatrixcolsep{3.3pc}\xymatrix{
X\ar@{=}[r] & S\tens X \ar[r]^-{\eta\tens \id_X} & X\tens Y\tens X
\ar[r]^-{\id\tens \epsilon} & X\tens S
\ar@{=}[r] & X
}
\]
is $\id_X$ and the composite
\[ 
\xymatrixcolsep{3.3pc}\xymatrix{
Y\ar@{=}[r] & Y\tens S \ar[r]^-{\id_Y\tens \eta} & Y\tens X\tens Y
\ar[r]^-{\epsilon \tens \id_Y } & S\tens Y
\ar@{=}[r] & Y
}
\]
is $\id_Y$.
We say that $Y$ is a \dfn{dual} for $X$, although it is
more precise to say that the dual is $(Y,\epsilon,\eta)$ since all
three pieces of structure are needed.
\end{defn}

\begin{remark}
If $Y$ is a dual for $X$, then there can be several choices for
$\epsilon$ and $\eta$ that serve as structure maps.  If one fixes $Y$
and $\epsilon$, however, then there is only one corresponding choice
for $\eta$; similarly, if one fixes $Y$ and $\eta$ then there is only
one choice for $\epsilon$.  This follows by the same argument that
shows that a functor can have at most one left (or right) adjoint.  
\end{remark}

The following result can be pulled out of the proof of \cite[Theorem III.1.6]{LMS}:

\begin{prop}
In a closed symmetric monoidal category
suppose that $X$ is dualizable with dual $(Y,\epsilon,\eta)$.  Then
the map $\tilde{\epsilon}\colon Y\ra X^*$, adjoint to $\epsilon$, is
an isomorphism.  Consequently,
$X^*$ is also a dual for $X$, with structure maps $\ev_X\colon
X^*\tens X\ra S$ and the composite
\[  S \llra{\eta} X\tens Y \llra{\id\tens \tilde{\epsilon}} X\tens
X^*. 
\]
\end{prop}

\begin{proof}
The duality axioms imply that the composite
\[ \cC(W,Y) \ra \cC(W\tens X,Y\tens X) \ra \cC(W\tens X,S)=\cC(W,X^*)
\]
is a bijection, for all objects $W$.  One readily checks that this
composite is induced by post-composition with the map
$\tilde{\epsilon}$ from the statement of the proposition.  
The Yoneda Lemma then yields that $\tilde{\epsilon}$
is an isomorphism.  Finally, one must check that 
\[ Y\tens X \llra{\tilde{\epsilon}\tens \id} X^*\tens X \llra{\ev_X}
S\]
equals $\epsilon$, but this is routine.  
\end{proof}

If $X$ is dualizable and $\ev_X\colon X^*\tens X\ra S$ and
$\cev_X\colon S\ra X\tens X^*$ satisfy the conditions of
Definition~\ref{de:dual} then we call $\cev_X$ the \dfn{coevaluation
  map} for $X$ (it is uniquely determined, of course).  The following
two results are standard:

\begin{prop}
In a closed symmetric monoidal category, an object $X$
is dualizable if and only if 
there exists a map $c$ that makes the following diagram commute:
\[ \xymatrix{
S\ar[d]_{\id_X}\ar[r]^-{c} & X\tens X^* \ar[d]^t \\
F(X,X) & X^*\tens X.\ar[l]^-\psi
}
\]
If $c$ exists, it is unique; and moreover, it is
precisely the coevaluation map for $X$.  
\end{prop}

\begin{proof}
See \cite[Theorem III.1.6]{LMS}.  The uniqueness of $c$ follows from
\cite[Proposition III.1.3]{LMS}, which shows that the horizontal map
$\psi$ is an isomorphism.
\end{proof}

\begin{prop} 
\label{pr:dualizable}
If $X$ and $Y$ are dualizable objects in a closed symmetric
monoidal category then the following are true:
\begin{enumerate}[(a)]
\item $X\tens Y$ and $X^*$ are dualizable;
\item $\psi\colon X\ra X^{**}$ is an isomorphism;
\item $\psi\colon X^*\tens Y^*\ra (X\tens Y)^*$ is an isomorphism.
\item $\cev_X\colon S\ra X\tens X^*$ is the composite
\[ \xymatrix{S\ar@{=}[r] & S^* \ar[r]^-{\ev_X^*} & (X^*\tens X)^* & X^{**}\tens
X^*\ar[l]^-\psi_-\iso
& X\tens X^* \ar[l]_-{\psi\tens \id}^-\iso
}
\]
\end{enumerate}
\end{prop}

\begin{proof}
Part (a) is elementary, while
parts (b) and (c) are from \cite[Proposition III.1.3]{LMS}.  For
part (d), perhaps the easiest method is to check that $\ev_X$ and the
given composite satisfy the properties of Definition~\ref{de:dual}.
To this end, consider the following diagram:
\[
\xymatrixcolsep{2.6pc}\xymatrix{
S^*\tens X \ar[r]^-{\ev_X^*\tens 1_X}\ar[dr]_\psi 
& (X^*\tens X)^*\tens X\ar[d]_\psi  & (X^{**}\tens X^*)\tens X
\ar[l]_-{\psi\tens 1}^-\iso &
(X\tens X^*)\tens X\ar[l]_-{\psi\tens 1\tens 1}^-\iso\ar[d]^{1\tens \ev_X} \\
& X^{**} && X\tens S\ar[ll]^\psi
}
\]
The vertical map labelled $\psi$ is the adjoint to the composite
\[\xymatrixcolsep{3pc}\xymatrix{
 (X^*\tens X)^*\tens X \tens X^* \ar[r]^{1\tens t}& (X^*\tens X)^*\tens
X^*\tens X \ar[r]^-{\ev_{X^*\tens X}} & S.
}
\]

We are required to show that the ``across-the-top, then down''
composition from $S^*\tens X$ to $X\tens S$ is the identity (after
canonical identifications of the domain and codomain with $X$).
But the triangle and the rectangle commute in any closed symmetric
monoidal category, by an easy verification (it suffices to check
commutativity in the category of finite-dimensional vector spaces over
a field, cf. \cite{HHP}). 
Since $\psi\colon X\ra X^{**}$ is an isomorphism by (b), this
completes the verification.  

The second condition from Definition~\ref{de:dual} is checked in a
similar manner.    The relevant diagram is a little easier:
\[ \xymatrixcolsep{2.2pc}\xymatrix{
X^*\tens S^* \ar[r]^-{1\tens \ev^*} \ar@/_3ex/[drrr]_{\id}& X^*\tens (X^*\tens X)^*
\ar[drr]_{\psi}
& X^*\tens (X^{**}\tens
X^*)\ar[l]_{1_{X^*}\tens \psi} & X^* \tens (X\tens X^*)\ar[l]_-{1\tens
\psi\tens 1} \ar[d]_{\ev_X\tens 1_{X^*}}\\
&&& S\tens X^*
}
\]  
The diagonal map labelled $\psi$ is the adjoint of the composite
\[\xymatrixcolsep{2.6pc}\xymatrix{
 X^*\tens (X^*\tens X)^* \tens X \ar[r]^-{t\tens 1}
& (X^*\tens X)^*\tens X^*\tens X \ar[r]^-{\ev_{X^*\tens X}} & S.
}
\]
The ``quadrilateral'' and ``triangle'' in the diagram again commute
in any closed symmetric monoidal category, and this completes the verification.
\end{proof}


\section{Leftover proofs}
\label{se:leftover}

\begin{proof}[Proof of Proposition~\ref{pr:Gysin->Cat}]
We include details
because several steps are a bit hard to remember, and this is the kind
of thing one wants to be able to just look up when needed.

For part (a), here is the check that $i_a$ is a right identity.  If
$x\in \cC_E(a,b)=E(b\times a)$ then 
\begin{align*}
x\circ i_a & = 
(\pi_{13})_!\Bigl (\pi_{12}^*x \cdot
\pi_{23}^*(\Delta^a_!(1))\Bigr ) \\
&=
(\pi_{13})_!\Bigl (\pi_{12}^*x \cdot
(\id_b\times \Delta^a)_!(1))\Bigr ) \qquad \ \qquad \qquad \ (\text{push-pull})\\
&= (\pi_{13})_! \Bigl ( (\id_b\times \Delta^a)_! \bigl ( (\id_b\times
\Delta^a)^*\pi_{12}^* x \cdot 1\bigr ) \Bigr ) \qquad\
(\text{projection formula}) \\
&= x.
\end{align*}
The last step used that $\pi_{13}\circ (\id_b\times
\Delta^a)=\id_{b\times a}$ and $\pi_{12}\circ (\id_b\times
\Delta^a)=\id_{b\times a}$.  
The verification that $i_a$ is a left identity is similar.

Write $\pi^{cba}_{ca}$ for the evident projection map $c\times b\times
a\ra c\times a$.  Let $x\in \cC_E(a,b)$, $y\in \cC_E(b,c)$, and $z\in
\cC_E(c,d)$.  
The proof of associativity proceeds by analyzing the
element
\[ \Omega= 
\bigl(\pi^{dcba}_{da}\bigr)_! \Bigl(
(\pi^{dcba}_{dc})^*(z)\cdot (\pi^{dcba}_{cb})^*(y)\cdot (\pi^{dcba}_{ba})^*(x)
\Bigr )
\]
in two different ways.  The first proceeds as follows:
\begin{align*}
\Omega &=
\bigl(\pi^{dca}_{da}\bigr)_! \bigl( \pi^{dcba}_{dca}\bigr)_! \Bigl[
(\pi^{dcba}_{dca})^*(\pi^{dca}_{dc})^*(z)\cdot 
(\pi^{dcba}_{cb})^*(y)\cdot (\pi^{dcba}_{ba})^*(x)
\Bigr ]   \\
&=
\bigl(\pi^{dca}_{da}\bigr)_!  \Bigl[
(\pi^{dca}_{dc})^*(z)\cdot 
(\pi^{dcba}_{dca})_! \bigl [
(\pi^{dcba}_{cb})^*(y)\cdot (\pi^{dcba}_{ba})^*(x) \bigr]
\Bigr ]  \qquad\text{(proj. form.)}\\
&=
\bigl(\pi^{dca}_{da}\bigr)_!  \Bigl[
(\pi^{dca}_{dc})^*(z)\cdot 
(\pi^{dcba}_{dca})_! (\pi^{dcba}_{cba})^*\bigl [
(\pi^{cba}_{cb})^*(y)\cdot (\pi^{cba}_{ba})^*(x) \bigr]
\Bigr ]  \\
&=
\bigl(\pi^{dca}_{da}\bigr)_!  \Bigl[
(\pi^{dca}_{dc})^*(z)\cdot 
(\pi^{dca}_{ca})^* (\pi^{cba}_{ca})_!\bigl [
(\pi^{cba}_{cb})^*(y)\cdot (\pi^{cba}_{ba})^*(x) \bigr]
\Bigr ]  \quad\text{(push-pull)}\\
&= z\cdot (y\cdot x).
\end{align*}
The first and third equalities just use functoriality.  For example,
in the third equality we use that
$\pi^{dbca}_{cb}=\pi^{cba}_{cb}\pi^{dcba}_{cba}$ and so forth.
We leave the reader to perform a similar series of steps to show that
$\Omega=(z\cdot y)\cdot x$.  This proves associativity, and so
finishes the proof of (a).

Part (b) is obvious.

For (c) we must show that if $f\colon a\ra b$ and $g\colon b\ra c$
then $R_g\circ R_f=R(gf)$.  That is, we must check the formula
\[ 
(\pi^{cba}_{ca})_!\Bigl [
(\pi^{cba}_{cb})^*(\id_c\times g)^*(i_c)\cdot
(\pi^{cba}_{ba})^*(\id_b\times f)^*(i_b)\Bigr ] = (\id_c\times
gf)^*(i_c).
\]
Note that the left side is $(\id_c\times g)^*(i_c)\circ (\id_b\times
f)^*(i_b)$.  

The first step is to use the two pullback squares 
\[\xymatrixcolsep{3.5pc}\xymatrix{
 c\times a \ar[d]_-{\pi_1\times f\times \pi_2} \ar[r]^{\pi_2} &
a \ar[d]^{f\times \id_a}\ar[r]^f & b\ar[d]^\Delta \\
c\times b\times a \ar[r]^{\pi^{cba}_{ba}} & b\times a
\ar[r]^{\id_b\times f} & b\times b
}
\]
to see that $(\pi^{cba}_{ba})^*(\id_b\times f)^*(i_b)=(\pi_1\times
f\times \pi_2)_!(1)$ (here we use that $\pi_2^*$ and $f^*$ are ring
maps and so send $1$ to $1$).  Next we compute that
\begin{align*}
R_g\circ R_f &=
(\pi^{cba}_{ca})_!\Bigl [
(\pi^{cba}_{cb})^*(\id_c\times g)^*(i_c)\cdot
(\pi^{cba}_{ba})^*(\id_b\times f)^*(i_b)\Bigr ] \\
&=
(\pi^{cba}_{ca})_!\Bigl [
(\pi^{cba}_{cb})^*(\id_c\times g)^*(i_c)\cdot 
(\pi_1\times f\times \pi_2)_!(1) \Bigr ]\\
&=(\pi^{cba}_{ca})_!(\pi_1\times f\times \pi_2)_!\Bigl [ (\pi_1\times
f\times \pi_2)^*(\pi^{cba}_{cb})^*(\id_c\times g)^*(i_c)\cdot 1\Bigr ]
\\
&= (\id_c\times gf)^*(i_c) \\
&= R(gf).
\end{align*}
In the second-to-last equality we have used that $\pi^{cba}_{ca}\circ
(\pi_1\times f\times \pi_2)=\id_{c\times a}$ and that
$(\id_c\times g)\pi^{cba}_{cb}(\pi_1\times f\times \pi_2)=\id_c\times
gf$.

To prove (d) 
we must verify that $i_a^*=i_a$ (for every
object $a$) and $(g\circ f)^*=f^*\circ g^*$ for every $f\in
\cC_E(a,b)$ and $g\in \cC_E(b,c)$. 
For the first of these, consider the twist map $t\colon a\times a\ra
a\times a$.
Since $t^2=\id_{a\times a}$ we have by Lemma~\ref{le:iso-shriek}
that $t_!=(t^*)^{-1}=t^*$.  So
\[ i_a^* = t^*(i_a)=t_!(i_a)=t_!(\Delta^a_!(1))=(t\circ
\Delta^a)_!(1)=\Delta^a_!(1)=i_a.\]

Write $t^{ab}_{ba}$ for the map $t\colon a\times b\ra b\times a$, and
similarly for other situations.  Then
\begin{align*} (g\circ f)^*&=
(t^{ac}_{ca})^*\Bigl[ (\pi^{cba}_{ca})_! [
(\pi^{cba}_{cb})^*(g)\cdot (\pi^{cba}_{ba})^*(f)]\Bigr ] \\
&=
(t^{ca}_{ac})_!\Bigl[ (\pi^{cba}_{ca})_! 
[
(\pi^{cba}_{cb})^*(g)\cdot (\pi^{cba}_{ba})^*(f)]\Bigr ] \\
&= (\pi^{cba}_{ac})_!\bigl [
(\pi^{cba}_{cb})^*(g)\cdot (\pi^{cba}_{ba})^*(f)\bigr ] \\
&= (\pi^{cba}_{ac})_! (t^{abc}_{cba})_! (t^{abc}_{cba})^* \bigl [
(\pi^{cba}_{cb})^*(g)\cdot (\pi^{cba}_{ba})^*(f)\bigr ] \\
&= (\pi^{abc}_{ac})_!   \bigl [
(\pi^{abc}_{cb})^*(g)\cdot (\pi^{abc}_{ba})^*(f)\bigr ] \\
&= (\pi^{abc}_{ac})_!   \bigl [
(\pi^{abc}_{bc})^*(t^{bc}_{cb})^*g\cdot
(\pi^{abc}_{ab})^*(t^{ab}_{ba})^*f\bigr ] \\
&= (\pi^{abc}_{ac})_!   \bigl [
(\pi^{abc}_{ab})^*(t^{ab}_{ba})^*f \cdot
(\pi^{abc}_{bc})^*(t^{bc}_{cb})^*g
\bigr ] \\
&= f^*\circ g^*.
\end{align*}
In the second and fourth equalities we have used
Lemma~\ref{le:iso-shriek}, but all of the other equalities use only
simple functoriality.

To prove the first part of (e) we argue as follows:
\begin{align*}
\alpha \circ R_f
& =
(\pi^{ZWY}_{ZY})_!\Bigl [ (\pi^{ZWY}_{ZW})^*(\alpha)
\cdot (\pi^{ZWY}_{WY})^*(\id_W\times f)^*(i_W)
\Bigr] \\
&= (\pi^{ZWY}_{ZY})_!\Bigl [ (\pi^{ZWY}_{ZW})^*(\alpha) \cdot
(\id_Z\times f\times \id_Y)_!(1) \Bigr] \\
&=(\pi^{ZWY}_{ZY})_! (\id_Z\times f\times \id_Y)_!\Bigl [
(\id_Z\times f\times \id_Y)^*(\pi^{ZWY}_{ZW})^*(\alpha) \cdot 1 \Bigr
] \\
&=\id_!\Bigl[ (\id_Z\times f)^*\alpha \cdot 1 \Bigr ]
\\
&= (\id_Z\times f)^*(\alpha).
\end{align*}
In the second equality we have used the push-pull axiom applied to the pullback diagram
\[ \xymatrixcolsep{3.5pc}\xymatrix{
Z\times W\times Y \ar[r]^{\pi^{ZWY}_{WY}} & W\times Y
\ar[r]^{\id_W\times f} & W\times
W \\
Z\times Y \ar[u]^{\id_Z\times f\times \id_Y} \ar[r]^{\pi_2} & Y
\ar[r]^f \ar[u]_{f\times\id_Y}& W.\ar[u]_\Delta
}
\]
In the third equality  we have used that $\pi^{ZWY}_{ZY}\circ (\id_Z\times
f\times \id_Y)=\id_{ZY}$ and $\pi^{ZWY}_{WY}\circ (\id_Z\times f\times
\id_Y)=\id_Z\times f$.  

The other parts of (e) are proven by similar arguments.  Part (f)
follows from (e) using
\[ I_f\circ R_q =I_f\circ i_B \circ R_q = (\id \times q)^*(I_f\circ
i_B)=(\id\times q)^*(f\times \id)^*(i_B)=(f\times q)^*(i_B).
\]  
The second part of (f) then follows using push-pull applied to the
square
\[ \xymatrix{
  A\times_B C \ar[r]\ar[d] & B\ar[d]^\Delta\\
 A\times C \ar[r]^{f\times q} & B\times B.
}
\]
Part (g) is similar to (f).

For (h), use that $R_p\circ I_g=\bigl( (p\times g)^Z_{XW}\bigr
)_!(1)=I_f\circ R_q$, by applying (f) and (g) together.  
Part (i) follows directly from (h) using the pullback diagram
\[ \xymatrix{
A \ar[r]^\id \ar[d]_\id & A \ar[d]^f \\
A \ar[r]^f & B.
}
\]
\end{proof}



\bibliographystyle{amsalpha}

\end{document}